\documentclass[a4paper]{article}
\usepackage{latexsym,bm}
\usepackage{amssymb,amsmath,amsthm}
\usepackage[english,german]{babel}
\usepackage{cite}
\usepackage[colorlinks=true]{hyperref}
\usepackage{color}

\newtheorem{theorem}{Theorem}[section]
\newtheorem{lemma}[theorem]{Lemma}
\newtheorem{remark}[theorem]{Remark}
\newtheorem{proposition}[theorem]{Propositon}
\newtheorem{definition}[theorem]{Definition}
\newtheorem{corollary}[theorem]{Corollary}
\numberwithin{equation}{section}
\hypersetup{linkcolor=blue,urlcolor=red,citecolor=red}
\title{Constrained approximate null
controllability of  coupled heat equation with periodic impulse controls\thanks{This work was supported by the
National Natural Science Foundation of China under grants 11771344, 11701138, 11601137.}}
\author{Lijuan Wang\thanks{School of
Mathematics and Statistics, Computational Science Hubei Key Laboratory,
Wuhan University, Wuhan 430072, China;
Email: ljwang.math@whu.edu.cn.}
\and
Qishu Yan\thanks{School of Science,
Hebei University of Technology, Tianjin 300400, China;  Email:
yanqishu@whu.edu.cn.}
\and
Huaiqiang Yu\thanks{School of Mathematics, Tianjin University, Tianjin 300354, China;  Email:  huaiqiangyu@tju.edu.cn.}
}

\date{}
\begin{document}
\selectlanguage{english}
\maketitle

\begin{abstract}
This paper is concerned with the constrained approximate null controllability of heat equation coupled by a real matrix $P$,  where the controls are
impulsive and periodically acted into the system through a series of real matrices  $\{Q_k\}_{k=1}^\hbar$.
The conclusions are given in two cases. In the case that the  controls act  globally into the system,  we prove that the system is  global constrained approximate null controllable
under a spectral condition of $P$ together with a rank condition of  $P$ and $\{Q_k\}_{k=1}^\hbar$; While in the case that  the controls act locally into the system,
we prove the global constrained approximate null controllability under a stronger condition for $P$ and the same rank condition as the above case. Moreover, we prove that the
above mentioned spectral condition of $P$ is necessary for global constrained approximate null controllability of  the control problem considered in this paper.
\end{abstract}
\vskip 10pt
    \noindent
\textbf{Keywords.}
Constrained approximate null controllability, Coupled heat equation, Impulse control
\vskip 10pt
    \noindent
\textbf{2010 Mathematics Subject Classifications.}
35K40, 93B05, 93C20

\section{Introduction}
   We start this section with some notations. Let $\mathbb{R}^+:=[0,+\infty),
\mathbb{N}:=\{0, 1, 2, \dots\}$ and $\mathbb{N}^+:=\{1, 2, \dots\}$. Let
$\Omega\subset \mathbb{R}^N (N\in \mathbb{N}^+)$ be a bounded domain with a
smooth boundary. Let $\hbar\in \mathbb{N}^+$ and $\omega_k\subset \Omega$
 $(1\leq k\leq \hbar)$ be an open and nonempty subset of $\Omega$ with $\omega:=\displaystyle{\cap_{k=1}^{\hbar}} \omega_k\not=\emptyset$. Let $P\in \mathbb{R}^{n\times n}$
and $Q_k\in \mathbb{R}^{n\times m}$ $(n, m\in \mathbb{N}^+, 1\leq k\leq \hbar)$. Denote by $\chi_{E}$ the characteristic function of the set $E\subset\mathbb R^n$.
Write $\triangle_n:=I_n\triangle=\mbox{diag}(\underbrace{\triangle,\triangle,\cdots,\triangle}_n)$, where
$I_n$ is the identity matrix in $\mathbb{R}^{n\times n}$ and
$\triangle$ is the Laplace operator with domain $D(\triangle):=H_0^1(\Omega)\cap H^2(\Omega)$. Define
\begin{equation*}\label{Intro-1}
    A:=\triangle_n+P\;\;\mbox{with}\;\;D(A):=(H_0^1(\Omega)\cap H^2(\Omega))^n.
\end{equation*}
One can easily check that $(A,D(A))$ generates an analytic semigroup $\{e^{At}\}_{t\in\mathbb{R}^+}$
over $(L^2(\Omega))^n$.  Let ${\mathcal L}((L^2(\Omega))^n;(L^2(\Omega))^n)$ denote the space of all linear bounded operators
from $(L^2(\Omega))^n$ to $(L^2(\Omega))^n$. Let $B_\delta(0)$ be the closed ball in $(L^2(\Omega))^n$ with
center at the origin and of radius $\delta>0$. Let
\begin{equation*}
\mathcal{M}_{\hbar}:=\{\{t_j\}_{j\in\mathbb{N}}:t_0=0<t_1<t_2<\cdots \textrm{ and }
t_{j+\hbar}-t_j=t_\hbar\;\;\mbox{for each}\;\; j\in \mathbb N\}.
\end{equation*}
Then for each $\{t_j\}_{j\in \mathbb{N}}\in \mathcal{M}_\hbar$, we have that
\begin{equation}\label{Intro-2}
t_{j+k\hbar}=t_j+t_{k\hbar}=t_j+kt_{\hbar}\;\;\mbox{for all}\;\;
j, k\in \mathbb{N}.
\end{equation}
    Throughout this paper, $C(\cdots)$ denotes a generic positive constant dependent on what
are enclosed in the bracket.
\subsection{Control problem}
   For arbitrarily fixed $\Lambda_\hbar=\{t_j\}_{j\in\mathbb{N}}\in \mathcal{M}_{\hbar}$,
we consider the following impulse controlled heat equation:
\begin{equation}\label{Intro-3}
\begin{cases}
  \displaystyle x'(t)=Ax(t),   & t\in\mathbb R^+\backslash \Lambda_\hbar,   \\
  x(t_j)=x(t_j^-)+\chi_{\omega_{\nu(j)}} Q_{\nu(j)} u_j,   &j\in\mathbb{N}^+,   \\
\end{cases}
\end{equation}
where  $u:=(u_j)_{j\in\mathbb N^+}\in l^\infty(\mathbb N^+;(L^2(\Omega))^m)$,
$x(t_j^-)$ denotes the left limit at $t_j$ for the function $x(\cdot)$,
$\nu(j):=j-[j/\hbar]\hbar$ for each $j\in\mathbb N^+$ and
$[s]:=\max\{k\in\mathbb N:\,k<s\}$ for each $s>0$.  It is clear that
\begin{equation}\label{Intro-4}
\nu(j+k\hbar)=\nu(j)=j\;\;\mbox{for each}\;\;k\in \mathbb{N}^+\;\;\mbox{and}\;\;
1\leq j\leq \hbar.
\end{equation}
Let $\mathcal{B}_\hbar=\{B_k\}_{k=1}^\hbar:=\{\chi_{\omega_k} Q_k\}_{k=1}^\hbar$.
Without loss of generality, we assume that $Q_k\not=0$ for each $1\leq k\leq \hbar$.
In the rest of this paper, we also denote the system (\ref{Intro-3}) by $[A, \mathcal{B}_\hbar, \Lambda_\hbar]$.
The adjoint operators of $A$ and $B_k$ ($1\leq k\leq\hbar$), denoted by $A^*$ and $B^*_k$ ($1\leq k\leq\hbar$) respectively,
can be expressed in the following manner:
\begin{equation}\label{Intro-4(1)}
A^*=\triangle_n+P^\top\;\;\mbox{and}\;\;B_k^*=\chi_{\omega_k} Q_k^\top,\;\;k=1,2,\cdots,\hbar.
\end{equation}
 We denote $\mathcal{B}_\hbar^*:=\{B_k^*\}_{k=1}^{\hbar}$ for simplicity. Here $P^\top$ and $Q_k^\top$ ($1\leq k\leq \hbar$) are the transpose of $P$ and $Q_k$ ($1\leq k\leq \hbar$), respectively.
For each $x_0\in (L^2(\Omega))^n$ and $u\in l^\infty(\mathbb N^+;(L^2(\Omega))^m)$,
we denote the unique solution to the system $[A,\mathcal{B}_{\hbar},\Lambda_{\hbar}]$
with initial state $x(0)=x_0$ as $x(\cdot;x_0,u,\Lambda_{\hbar})$.
Define the control constraint set as follows:
\begin{equation*}
\mathcal U:=\{u=(u_j)_{j\in\mathbb{N}^+}
\in l^\infty(\mathbb N^+;(L^2(\Omega))^m):\, \|u_j\|_{(L^2(\Omega))^m}\leq 1\;\;\mbox{for each}\;\;j\in\mathbb{N}^+\}.
\end{equation*}

In order to present the main results of this paper, we introduce the following definition.
\begin{definition}\label{Intro-5}
    Let $\Lambda_{\hbar}\in \mathcal{M}_{\hbar}$ be fixed.
\begin{enumerate}
  \item[(i)] Let $\varepsilon>0$ be fixed. The system $[A,\mathcal{B}_{\hbar},\Lambda_{\hbar}]$
  is called $\varepsilon$-global constrained approximate null controllable
      (denoted by $(\varepsilon$-$\text{GCAC})_{\hbar}$ later), if
      for any $x_0\in (L^2(\Omega))^n$, there exists a control $u\in \mathcal U$ and $k\in\mathbb{N}$
    so that $x(t_k;x_0,u,\Lambda_\hbar)\in B_{\varepsilon}(0)$.
   \item[(ii)]
     The system $[A,\mathcal{B}_{\hbar},\Lambda_{\hbar}]$ is called global constrained approximate null controllable
      (denoted by $(\text{GCAC})_{\hbar}$ later)
      if $[A,\mathcal{B}_{\hbar},\Lambda_{\hbar}]$ is $(\varepsilon$-$\text{GCAC})_{\hbar}$ for each $\varepsilon>0$.
\end{enumerate}
\end{definition}

\subsection{Main results}
    Let $0<\lambda_1<\lambda_2\leq \lambda_3\leq \cdots$ be the eigenvalues of $-\triangle$ (see \cite{Evans-2}),
$e_i\in D(\triangle)$
be the normal eigenfunction of $-\triangle$ with respect to $\lambda_i$ for each $i\in\mathbb{N}^+$, and $\sigma(P)$
be the spectrum of $P$. The first main result of this paper reads as follows.
\begin{theorem}\label{Intro-6}
   Let $\Lambda_{\hbar}\in\mathcal{M}_{\hbar}$ be  fixed.
   Assume that there is a $k^*\in \mathbb{N}^+$ so that
   \begin{equation}\label{yu-3-6-2}
   \mbox{rank}\;( e^{-P t_1}Q_{\nu(1)},
  e^{-Pt_2}Q_{\nu(2)},\cdots,e^{-P t_{k^*}}Q_{\nu(k^*)})=n.
  \end{equation}
\begin{enumerate}
  \item [(i)]  If $\omega=\Omega$ and $\sigma(P)\subset \{\rho\in\mathbb{C}:\mbox{Re}(\rho) \leq \lambda_1\}$,
  then the system $[A, \mathcal{B}_\hbar, \Lambda_\hbar]$ is $(\text{GCAC})_{\hbar}$.
  \item[(ii)] If $\omega \subsetneq \Omega$ and
  $\langle P\eta,\eta\rangle_{\mathbb{R}^n}\leq \lambda_1\|\eta\|^2_{\mathbb{R}^n}$ for each $\eta\in\mathbb{R}^n$,
  then the system $[A, \mathcal{B}_\hbar, \Lambda_\hbar]$ is $(\text{GCAC})_{\hbar}$.
  \end{enumerate}
\end{theorem}
    The proof of Theorem \ref{Intro-6} depends heavily on the following proposition:
\begin{proposition}\label{Preli-2}
    Let $\Lambda_{\hbar}\in\mathcal{M}_{\hbar}$ and $k\in\mathbb{N}^+$ be fixed. The following two claims are equivalent:
\begin{enumerate}
  \item [(i)] $\text{rank}\left(e^{-P t_1}Q_{\nu(1)},e^{-P t_2}Q_{\nu(2)},\cdots, e^{-Pt_k}Q_{\nu(k)}\right)=n$.
  \item [(ii)] There are two constants $\theta\in(0,1)$ (independent of $k$) and $C(k)>0$ so that
\begin{eqnarray}\label{Preli-2:1}
 &\;&\|e^{A^*t_{k+1}}z\|_{(L^2(\Omega))^n}\nonumber\\
 &\leq& C(k)\left(\displaystyle{\sum_{j=1}^k}
    \|B^*_{\nu(j)}e^{A^*(t_{k+1}-t_j)}z\|_{(L^2(\Omega))^m)}\right)^{\theta}
    \|z\|^{1-\theta}_{(L^2(\Omega))^n}\;\;\mbox{for all}\;\;z\in (L^2(\Omega))^n.
\end{eqnarray}
\end{enumerate}
\end{proposition}

\par The interpolation inequality (\ref{Preli-2:1}) in Proposition \ref{Preli-2} is a quantitative form of unique continuation
    for the coupled heat equation.  As we know, it is new.
    The proof of Proposition  \ref{Preli-2} will be given in Section 2.

The second main result of this paper is as follows.
\begin{theorem}\label{20-2-28}
 If there is a $\rho\in \sigma(P)$ so that $\mbox{Re}(\rho)>\lambda_1$, then for each
  $\varepsilon>0$ and $\Lambda_{\hbar}\in\mathcal{M}_{\hbar}$, the system $[A, \mathcal{B}_\hbar, \Lambda_\hbar]$ is not $(\varepsilon$-$\text{GCAC})_{\hbar}$.
\end{theorem}
   In \cite{Q-W-Y}, the authors stressed that the controlled system (\ref{Intro-3}) can be understood by two ways:
\begin{enumerate}
  \item [(W1)]With $\hbar$ controllers $\{B_k\}_{k=1}^{\hbar}$ and impulse instants $\{t_j\}_{j\in\mathbb{N}}\in\mathcal{M}_{\hbar}$ in hands, we put periodically the controllers into the system $x'=Ax$ at impulse instants.
  \item [(W2)] With $\hbar$ controllers  $\{B_k\}_{k=1}^{\hbar}$  in hands,
we first choose impulse instants $\{t_j\}_{j\in\mathbb{N}}\in\mathcal{M}_{\hbar}$, and put the controllers periodically into the system $x'=Ax$ at the impulse instants.
\end{enumerate}
\par
    The third main result of this paper is as follows.
\begin{theorem}\label{Intro-7}
   Assume that rank ($\lambda I_n-P, Q_1, \cdots, Q_\hbar$)$=n$ for each $\lambda\in \mathbb{C}$.
\begin{enumerate}
  \item [(i)]  If $\omega=\Omega$ and $\sigma(P)\subset \{\rho\in\mathbb{C}:\mbox{Re}(\rho) \leq \lambda_1\}$,
  then there is $\Lambda_\hbar\in \mathcal{M}_\hbar$ so that
  the system $[A, \mathcal{B}_\hbar, \Lambda_\hbar]$ is $(\text{GCAC})_{\hbar}$.
  \item[(ii)] If $\omega \subsetneq \Omega$ and
  $\langle P\eta,\eta\rangle_{\mathbb{R}^n}\leq \lambda_1\|\eta\|^2_{\mathbb{R}^n}$ for each $\eta\in\mathbb{R}^n$,
  then there is $\Lambda_\hbar\in \mathcal{M}_\hbar$ so that
  the system $[A, \mathcal{B}_\hbar, \Lambda_\hbar]$ is $(\text{GCAC})_{\hbar}$.
\end{enumerate}
\end{theorem}
     It is clear that Theorem \ref{Intro-6} and Theorem \ref{20-2-28}
    correspond to (W1), and Theorem \ref{Intro-7} corresponds to (W2).
\subsection{Some comments}
    Some comments are listed as follows.
\begin{itemize}
  \item Constrained controllability is a very interesting topic in control theory.
  Among the existing literature on this topic for the evolution system, the control inputs are usually distributed in the whole time interval,
  i.e., they may affect the control system at each instant of time
  (see, for instance, \cite{Ahmed,Barmish-Schmitendorf,Carja,Evans,Narukawa-1981,Narukawa-1982,Pandolfi,Peichl-Schappacher,Schmitendorf-Barmish,Son-Su}).
   We call them distributed control for simplicity. However, in many cases, impulse control can give an efficient way to deal with systems,
  which can not endure distributed control inputs, or in some applications, it is impossible to provide distributed control inputs.
  Impulse control systems have attracted the attention of many researchers
   (see, for instance, \cite{Duan-Wang-Zhang,Duan-Wang,Q-W,Q-W-Y,Trelet-wang-zhang}). In these works, controllability, stabilizability and optimal control problems of impulse control systems were  studied.  Especially, in \cite{Q-W}, when $\omega_k=\omega$ for each $1\leq k\leq \hbar$, the authors proved that the null controllability of coupled heat equation with impulse controls did not hold, except for the case $\omega=\Omega$. Indeed, for the impulse controlled  coupled heat equation, when
   $\omega\subsetneq\Omega$ and $\omega_k=\omega$ for each $1\leq k\leq \hbar$, we can only expect the approximate controllability (see Theorem 1.3 in \cite{Q-W}). Thus, from  the viewpoint of controllability,  impulse control and distributed control are
   intrinsically different in infinite dimensional setting (for the distributed control case, for example, one can refer to \cite{A-B-D-G}). This motivates us to study the constrained approximate null controllability of the coupled heat equation with impulse controls.

    \item Let $\Lambda_\hbar\in \mathcal{M}_\hbar$ be fixed. The system $[A,\mathcal{B}_{\hbar},\Lambda_{\hbar}]$ is $(\text{GCAC})_\hbar$ if and only if the following
    {\textbf{Claim (H)}} holds.
\vskip 5pt
\begin{itemize}
\item[\textbf{(H)}\quad] \emph{If for  each $\varepsilon>0$ and each $x_0\in (L^2(\Omega))^n$, there exists a control $u\in \mathcal{U}$ and $t\in \mathbb{R}^+$ so that
                       $x(t;x_0,u,\Lambda_{\hbar})\in B_\varepsilon(0)$.}
\end{itemize}
    This conclusion is proved in Appendix A of Section 6.
\item When $\omega=\Omega$ (i.e., the control acts globally into the system $[A,\mathcal{B}_{\hbar},\Lambda_{\hbar}]$), we can study the constrained null
controllability of the system $[A,\mathcal{B}_{\hbar},\Lambda_{\hbar}]$.
\emph{The system  $[A,\mathcal{B}_{\hbar},\Lambda_{\hbar}]$ is called constrained null
controllable if for each $x_0\in (L^2(\Omega))^n$, there is a control $u\in\mathcal{U}$ and $k\in\mathbb{N}$ so that $x(t_k;x_0,u,\Lambda_{\hbar})=0$.} For the constrained null
controllability of the system $[A,\mathcal{B}_{\hbar},\Lambda_{\hbar}]$, we can claim that
\begin{corollary}\label{yu-proposition-b-1}
    Suppose that $\omega=\Omega$ and there exists a $k^*\in\mathbb{N}^+$ so that (\ref{yu-3-6-2}) holds. Then, the system $[A,\mathcal{B}_{\hbar},\Lambda_{\hbar}]$ is constrained null controllable if and only if $\sigma(P)\subset\{\rho\in\mathbb{C}:\mbox{Re}(\rho)\leq \lambda_1\}$.
\end{corollary}
     Under the assumptions of Corollary \ref{yu-proposition-b-1},
      it follows from Corollary \ref{yu-proposition-b-1}, $(i)$ in Theorem \ref{Intro-6} and Theorem \ref{20-2-28} that the system $[A,\mathcal{B}_{\hbar},\Lambda_{\hbar}]$ is
     constrained null controllable if and only if it is $(\text{GCAC})_{\hbar}$.
     In Appendix B of Section 6, we will present the proof of Corollary \ref{yu-proposition-b-1}.

    \item
  Recall $P, \{Q_k\}_{1\leq k\leq \hbar}$ and $\Lambda_\hbar=\{t_j\}_{j\in\mathbb N}\in \mathcal M_\hbar$
(see the beginning of this paper). Consider the following impulse controlled ordinary differential equation:
\begin{equation}\label{yu-3-6-1}
\begin{cases}
    y'(t)=Py(t),&t\in\mathbb{R}^+\setminus\Lambda_{\hbar},\\
    y(t_j)=y(t_j^-)+Q_{\nu(j)}v_j,&j\in\mathbb{N}^+,
\end{cases}
\end{equation}
    where $v:=(v_j)_{j\in\mathbb{N}^+}\in l^\infty(\mathbb{N}^+;\mathbb{R}^m)$ and
    $y(\cdot):\mathbb{R}^+\to \mathbb{R}^n$. For each $y_0\in\mathbb{R}^n$ and $v\in l^\infty(\mathbb{N}^+;\mathbb{R}^m)$, we denote the solution of (\ref{yu-3-6-1}) by $y(\cdot;y_0,v,\Lambda_{\hbar})$ with $y(0)=y_0$. Define the control constraint set:
$$
    \mathcal{W}:=\{v=(v_j)_{j\in\mathbb{N}^+}\in l^\infty(\mathbb{N}^+;\mathbb{R}^m):\|v_j\|_{\mathbb{R}^m}\leq 1\;\mbox{for each}\;
    j\in\mathbb{N}^+\}.
$$
    \emph{If for each $y_0\in\mathbb{R}^n$, there is a control $v\in\mathcal{W}$ and $k\in\mathbb{N}$ so that
    $y(t_k;y_0,v,\Lambda_{\hbar})=0$, then we call the system (\ref{yu-3-6-1}) is constrained null controllable.}     Using the idea in \cite[Theorem 2.6]{Evans} and by  similar arguments as those to prove Proposition \ref{PRO-1} and Theorem \ref{20-2-28}, we can show the following result:
\vskip 5pt
\emph{Assume that there is a $k^*\in \mathbb{N}^+$ so that (\ref{yu-3-6-2}) holds. Then the system (\ref{yu-3-6-1}) is constrained null controllable if and only if $\sigma(P)\subset\{\rho\in\mathbb{C}:\mbox{Re}(\rho)\leq 0\}$.}
\item
  We expect to improve $(ii)$ of Theorem~\ref{Intro-6} as follows: If $(a)$ there is a $k^*\in \mathbb{N}^+$ so that (\ref{yu-3-6-2}) holds; $(b)$ $\omega\subsetneq\Omega$; $(c)$ $\sigma(P)\subset \{\rho\in\mathbb{C}:\mbox{Re}(\rho) \leq \lambda_1\}$,
  then the system $[A, \mathcal{B}_\hbar, \Lambda_\hbar]$ is $(\text{GCAC})_{\hbar}$. Unfortunately, we do not know whether this claim is true or not.  Indeed, as we know, for the distributed control in infinite dimensional setting of abstract framework,
  when the control acts locally into the system, how to present the corresponding spectral condition for constrained controllability is also open.


\item In Theorem \ref{Intro-6}, the assumption (\ref{yu-3-6-2}) is necessary, except for the case
   $\sigma(P)\subset \{\rho\in\mathbb{C}:\mbox{Re}(\rho) < \lambda_1\}$.
   Indeed, when $\sigma(P)\subset \{\rho\in\mathbb{C}:\mbox{Re}(\rho) < \lambda_1\}$, we can show easily that
   there are $\mu>0$ and $C>0$ so that  $\|e^{At}\|_{\mathcal{L}((L^2(\Omega))^n;(L^2(\Omega))^n)}
   \leq Ce^{-\mu t}$ for each $t\in\mathbb{R}^+$,  which means the system $[A,\mathcal{B}_{\hbar},\Lambda_{\hbar}]$ is
   $(\text{GCAC})_{\hbar}$ with null control. However, if $(a)$ $\langle P\eta,\eta\rangle_{\mathbb{R}^n}\leq \lambda_1\|\eta\|^2_{\mathbb{R}^n}$ for each $\eta\in\mathbb{R}^n$; $(b)$ there is a $\rho\in \sigma(P)$ so that
   $\mbox{Re}(\rho)=\lambda_1$; $(c)$ (\ref{yu-3-6-2}) does not hold,  then we can provide an example to show
   that the system $[A,\mathcal{B}_{\hbar},\Lambda_{\hbar}]$ is not $(\text{GCAC})_{\hbar}$. This example
   is given in Appendix C of Section 6.

\item
    In our paper, we assume that $\displaystyle{\cap_{k=1}^{\hbar}} \omega_k\not=\emptyset$.
    We only used this assumption in the proof of Proposition \ref{Preli-2}. For the case of $\displaystyle{\cap_{k=1}^{\hbar}} \omega_k=\emptyset$, we do not know whether
    Proposition \ref{Preli-2} holds or not.
\end{itemize}

\subsection{Plan of this paper}
   This paper is organized as follows. In Section 2, we present some preliminary results.
In Sections 3, 4 and 5, we give proofs of Theorem~\ref{Intro-6}, \ref{20-2-28} and \ref{Intro-7},
respectively. In Section 6, we show some appendices.

\section{Preliminaries}
    We first recall the following interpolation inequality for scalar-valued heat equation.
\begin{lemma}\label{Preli-1} (\cite{Phung-Wang-Zhang})
    Let $\{e^{\triangle t}\}_{t\in\mathbb{R}^+}$ be the analytic semigroup generated by $\triangle$
     with domain $H_0^1(\Omega)\cap H^2(\Omega)$. Let $0<t<T<+\infty$
      and $\widetilde\omega\subset \Omega$ be
     a nonempty open subset.  Then there are two constants $\widetilde\theta:=\widetilde \theta(\Omega,\widetilde\omega,T)\in (0,1)$
     and $C(T,t,\Omega,\widetilde\omega)>0$ so that
\begin{equation*}
    \|e^{\triangle t}f\|_{L^2(\Omega)}\leq C(T,t,\Omega,\widetilde\omega)\|\chi_{\widetilde\omega}e^{\triangle t} f\|^{\widetilde\theta}_{L^2(\Omega)}\|f\|_{L^2(\Omega)}^{1-\widetilde\theta}\;\;\mbox{for all}\;\;f\in L^2(\Omega).
\end{equation*}
\end{lemma}
    The next lemma presents the equivalence between a rank condition
    and an observability inequality for finite dimensional system.
   \begin{lemma}\label{Preli-11}
 Let  $k\in\mathbb{N}^+$  and $0<\tau_1<\tau_2<\cdots<\tau_k<+\infty$ be fixed. Let  $\widetilde{P}\in\mathbb{R}^{n\times n}$
and $\{\widetilde Q_j\}_{j=1}^k\subset\mathbb{R}^{n\times m}$.
The following two claims are equivalent:
\begin{enumerate}
  \item [(i)] $\text{rank}\;\left(e^{-\widetilde P \tau_1}\widetilde Q_{1},e^{-\widetilde P \tau_2}\widetilde Q_{2},\cdots, e^{-\widetilde P\tau_k}\widetilde Q_{k}\right)=n$.
  \item[(ii)]  There is a constant $C(k)>0$ so that
\begin{equation}\label{20-2-12}
    \|v\|^2_{\mathbb{R}^n}\leq C(k)\sum_{j=1}^k \|\widetilde Q_{j}^\top e^{-\widetilde P^{\top}\tau_j}v\|^2_{\mathbb{R}^m}
    \;\;\mbox{for each}\;\;v\in \mathbb{R}^n.
\end{equation}
\end{enumerate}
\end{lemma}
    \noindent Lemma \ref{Preli-11} is easy to prove. However, for the sake of completeness, we give its detailed proof
    in Appendix D of Section 6.
%
%
%

Based on Lemma~\ref{Preli-1} and Lemma~\ref{Preli-11}, we present the proof of  Proposition \ref{Preli-2}.

\begin{proof}[Proof of Proposition \ref{Preli-2}] The proof is split into the following two steps.
\vskip 5pt
\emph{Step 1. We show that $(ii)\Rightarrow(i)$.}

By contradiction, there would be a $v^*\in\mathbb{R}^n\setminus\{0\}$ so that
\begin{equation}\label{Preli-2:2}
    Q^{\top}_{\nu(j)}e^{P^\top(t_{k+1}-t_j)}v^*=0\;\;\mbox{for all}\;\;j\in\{1,2,\ldots,k\}.
\end{equation}
Recalling the first equality in (\ref{Intro-4(1)}), we have that
\begin{equation}\label{Preli-2:3}
e^{A^* t}=e^{(\triangle_n+P^\top)t}=e^{\triangle_n t} e^{P^\top t}\;\;\mbox{for all}\;\;t\in\mathbb R^+,
\end{equation}
and
\begin{equation}\label{20-2-13}
e^{\triangle_n s} e^{P^\top t}=e^{P^\top t} e^{\triangle_n  s},\;\;e^{\triangle_m s} Q^\top_{\nu(j)}= Q^\top_{\nu(j)}e^{\triangle_n  s}\;\;\mbox{for all}\;\;
t, s\in \mathbb{R}^+\textrm{ and}\;\;j\in\mathbb N^+,
\end{equation}
where $\triangle_m:=I_m\triangle=\mbox{diag}(\underbrace{\triangle,\triangle,\cdots,\triangle}_m)$.
Take $f\in L^2(\Omega)\setminus\{0\}$ and let $z^*=v^*f$. It is clear that
$z^*\in (L^2(\Omega))^n$. According to \eqref{Preli-2:2}-(\ref{20-2-13}),
for each $j\in\{1,2,\dots,k\}$,
\begin{equation*}\label{Preli-2:4}
\begin{array}{lll}
    B^*_{\nu(j)}e^{A^*(t_{k+1}-t_j)}z^*&=&
    \chi_{\omega_{\nu(j)}}Q_{\nu(j)}^{\top} e^{P^\top(t_{k+1}-t_j)} e^{\triangle_n (t_{k+1}-t_j)} z^*\\
    &=&\chi_{\omega_{\nu(j)}}e^{\triangle(t_{k+1}-t_j)}fQ_{\nu(j)}^{\top}e^{P^\top(t_{k+1}-t_j)}v^*=0.
\end{array}
\end{equation*}
This, along with (\ref{Preli-2:1}), implies that $e^{A^* t_{k+1}}z^*=0$, which, combined with
(\ref{Preli-2:3}) again, indicates that
\begin{equation}\label{Preli-2:5}
    e^{\triangle t_{k+1}} f e^{P^\top t_{k+1}}v^*=0.
\end{equation}
Since $f\neq 0$, by the backward uniqueness of heat equation and (\ref{Preli-2:5}),
we have that $e^{P^\top t_{k+1}}v^*=0$. Then $v^*=0$.
It contradicts to the fact that  $v^*\in\mathbb{R}^n\setminus\{0\}$.
Hence, $(i)$ holds.
\vskip 5pt
\emph{Step 2. We prove that $(i)\Rightarrow(ii)$.}

Indeed, according to Lemma~\ref{Preli-1}, for each $T>0$ and $t\in(0,T)$,
there are two constants $C(T,t,\Omega,\omega)>0$ and ${\theta}:={\theta}(\Omega,\omega,T)\in (0,1)$ so that
\begin{equation}\label{20-2-24}
    \|e^{\triangle t}f\|_{L^2(\Omega)}\leq C(T,t,\Omega,\omega)\|\chi_{\omega}e^{\triangle t} f\|^{{\theta}}_{L^2(\Omega)}\|f\|_{L^2(\Omega)}^{1-{\theta}}\;\;\mbox{for all}\;\;f\in L^2(\Omega).
\end{equation}
  Noting that
$$
\|e^{\triangle T}f\|_{L^2(\Omega)}\leq C(T,t)\|e^{\triangle t}f\|_{L^2(\Omega)} \,\textrm{ for each }t\in[0,T],
$$
by (\ref{20-2-24}) and Young's inequality, for each $T>0$, $t\in(0,T)$ and $\varepsilon>0$, we have that
\begin{equation}\label{Preli-2:8}
    \|e^{\triangle T}f\|^2_{L^2(\Omega)}\leq C(T,t,\Omega,\omega)
    \varepsilon^{-\gamma}\|\chi_\omega e^{\triangle t} f\|^2_{L^2(\Omega)}
    +\varepsilon\|f\|^2_{L^2(\Omega)}\;\;\mbox{for all}\;\;f\in L^2(\Omega),
\end{equation}
where $\gamma:=(1-{\theta})/{\theta}$.

We fix  $z\in (L^2(\Omega))^n\setminus \{0\}$ arbitrarily (When $z=0$, (\ref{Preli-2:1}) is obvious).
On one hand,  by (\ref{Preli-2:3}), \eqref{20-2-13} and Lemma~\ref{Preli-11} (where  $\{\tau_j\}_{j=1}^k$, $\widetilde P$ and
 $\{\widetilde Q_j\}_{j=1}^k$ are replaced by  $\{t_j\}_{j=1}^k$, $P$ and $\{Q_{\nu(j)}\}_{j=1}^k$, respectively),
 there exists a constant $C_1(k)>0$ so that
\begin{eqnarray}\label{Preli-2:10}
\|e^{A^*t_{k+1}}z\|^2_{(L^2(\Omega))^n}
&=&\|e^{P^\top t_{k+1}}e^{\triangle_n t_{k+1}}z\|_{(L^2(\Omega))^n}^2\nonumber\\
&\leq &C_1(k) \sum_{j=1}^k\|Q_{\nu(j)}^\top e^{-P^\top t_j}e^{P^\top t_{k+1}}e^{\triangle_n t_{k+1}}z\|_{(L^2(\Omega))^m}^2\nonumber\\
&=& C_1(k) \sum_{j=1}^k\|e^{\triangle_m t_{k+1}}Q_{\nu(j)}^\top e^{P^\top (t_{k+1}-t_j)}z\|_{(L^2(\Omega))^m}^2.
\end{eqnarray}
Furthermore, it follows from  \eqref{Preli-2:8} (where $T=t_{k+1}$ and $t=t_{k+1}-t_j$, $1\leq j\leq k$) that for each $\varepsilon>0$,
\begin{eqnarray*}
&&\|e^{\triangle_m t_{k+1}}Q_{\nu(j)}^\top e^{P^\top (t_{k+1}-t_j)}z\|_{(L^2(\Omega))^m}^2\\
&\leq& C(t_{k+1},t_{k+1}-t_j,\Omega,\omega)
\varepsilon^{-\gamma}\|\chi_\omega e^{\triangle_m (t_{k+1}-t_j)}
Q_{\nu(j)}^\top e^{P^{\top}(t_{k+1}-t_j)}z\|_{(L^2(\Omega))^m}^2\\
&\;&
+\varepsilon\|Q_{\nu(j)}^\top e^{P^{\top}(t_{k+1}-t_j)}z\|_{(L^2(\Omega))^m}^2.
\end{eqnarray*}
This, along with \eqref{Preli-2:10}, implies that
\begin{eqnarray}\label{Preli-2:13}
&\;&\|e^{A^*t_{k+1}}z\|^2_{(L^2(\Omega))^n}\nonumber\\
    &\leq&C_1(k)
\sum_{j=1}^k \Big(C(t_{k+1},t_{k+1}-t_j,\Omega,\omega)
\varepsilon^{-\gamma}\|\chi_\omega e^{\triangle_m (t_{k+1}-t_j)}
Q_{\nu(j)}^\top e^{P^{\top}(t_{k+1}-t_j)}z\|_{(L^2(\Omega))^m}^2\nonumber\\
&\;&\;\;\;\;\;\;\;\;\;\;\;\;\;\;\;\;\;\;\;
\;\;\;\;\;\;
+\varepsilon\|Q_{\nu(j)}^\top e^{P^{\top}(t_{k+1}-t_j)}z\|_{(L^2(\Omega))^m}^2\Big).
\end{eqnarray}
On the other hand, by (\ref{Preli-2:3}) and \eqref{20-2-13}, for each $1\leq j\leq k$, we have that
\begin{eqnarray}\label{zpreli-2:14}
&\;&\|\chi_\omega e^{\triangle_m (t_{k+1}-t_j)}
Q_{\nu(j)}^\top e^{P^{\top}(t_{k+1}-t_j)}z\|^2_{(L^2(\Omega))^m}\nonumber\\
&\leq&\|\chi_{\omega_{\nu(j)}}Q_{\nu(j)}^\top e^{\triangle_n(t_{k+1}-t_j)}
    e^{P^{\top}(t_{k+1}-t_j)}z\|_{(L^2(\Omega))^m}^2\nonumber\\
 &=&\|B^*_{\nu(j)} e^{A^*(t_{k+1}-t_j)}z\|_{(L^2(\Omega))^m}^2.
\end{eqnarray}
Moreover, it is clear that  there exists a constant $C_2(k)>0$ so that
\begin{eqnarray}
\sum_{j=1}^k\|Q_{\nu(j)}^\top e^{P^{\top}(t_{k+1}-t_j)}z\|_{(L^2(\Omega))^m}^2
\leq C_2(k)\|z\|_{(L^2(\Omega))^n}^2.\label{Preli-2:15}
\end{eqnarray}
It follows from (\ref{Preli-2:13})-(\ref{Preli-2:15}) that
\begin{eqnarray}\label{Preli-2:16}
&\;&\|e^{A^*t_{k+1}}z\|_{(L^2(\Omega))^n}^2\nonumber\\
&\leq& C_1(k)C_2(k)\Big(
    C_3(k)\varepsilon^{-\gamma}\displaystyle{\sum_{j=1}^k}
    \|B^*_{\nu(j)}e^{A^*(t_{k+1}-t_j)}z\|_{(L^2(\Omega))^m}^2
    +\varepsilon\|z\|_{(L^2(\Omega))^n}^2\Big),
\end{eqnarray}
where $C_3(k):=C_2^{-1}(k)\displaystyle{\max_{1\leq j\leq k}}C(t_{k+1},t_{k+1}-t_j,\Omega,\omega)$.

We next claim that
\begin{equation*}
\sum_{j=1}^k\|B^*_{\nu(j)}e^{A^*(t_{k+1}-t_j)}z\|_{(L^2(\Omega))^m}^2\not=0.
\end{equation*}
Otherwise, according to (\ref{Preli-2:16}),
\begin{equation*}
\|e^{A^*t_{k+1}}z\|_{(L^2(\Omega))^n}^2\leq  \varepsilon C_1(k) C_2(k)\|z\|_{(L^2(\Omega))^n}^2.
\end{equation*}
This, along with the arbitrariness of $\varepsilon$, implies that $e^{A^*t_{k+1}}z=0$, which,
combined with (\ref{Preli-2:3}) and the backward uniqueness of heat equation, indicates that
$z=0$. It contradicts to the fact that $z\not=0$.

Let
\begin{equation*}
\varepsilon:=\Big(C_3(k)\displaystyle{\sum_{j=1}^k}\|B^*_{\nu(j)}e^{A^*(t_{k+1}-t_j)}z\|^2_{(L^2(\Omega))^m}\Big/
    \|z\|_{(L^2(\Omega))^n}^2\Big)^{\frac{1}{\gamma+1}}.
\end{equation*}
By (\ref{Preli-2:16}), we conclude  that
\begin{eqnarray}\label{Preli-2:17}
&\;&\|e^{A^*t_{k+1}}z\|_{(L^2(\Omega))^n}^2\nonumber\\
&\leq& 2 C_1(k)C_2(k) C_3^{\frac{1}{\gamma+1}}(k)
    \Big(\displaystyle{\sum_{j=1}^k}\|B^*_{\nu(j)}e^{A^*(t_{k+1}-t_j)}
    z\|^2_{(L^2(\Omega))^m}\Big)^{\frac{1}{\gamma+1}}
    \|z\|_{(L^2(\Omega))^n}^{\frac{2\gamma}{\gamma+1}}\nonumber\\
&\leq& 2 C_1(k)C_2(k) C_3^{\frac{1}{\gamma+1}}(k)
    \Big(\displaystyle{\sum_{j=1}^k}\|B^*_{\nu(j)}e^{A^*(t_{k+1}-t_j)}
    z\|_{(L^2(\Omega))^m}\Big)^{\frac{2}{\gamma+1}}
    \|z\|_{(L^2(\Omega))^n}^{\frac{2\gamma}{\gamma+1}}.
\end{eqnarray}
Since $\gamma=(1-\theta)/\theta$,
(\ref{Preli-2:1}) follows from (\ref{Preli-2:17}) with
$C(k):=\sqrt{2 C_1(k)C_2(k) C_3^{\frac{1}{\gamma+1}}(k)}$.

Hence,  we finish the proof of Step 2.
\par

In summary, we complete the proof of this proposition.
\end{proof}

The following definition plays an important role in the proof of Theorem~\ref{Intro-6}.
\begin{definition}\label{Preli-3}
Let $\Lambda_{\hbar}\in \mathcal{M}_{\hbar}$ be fixed.
Let $\delta>0$ and $k\in\mathbb{N}^+$.
If there is a constant $D(k,\delta)\geq 0$ so that for all $z\in (L^2(\Omega))^n$,
\begin{equation}\label{Preli-3:1}
\|e^{A^*t_{k}}z\|_{(L^2(\Omega))^n}
\leq D(k,\delta)\sum_{j=1}^{k}
\|B^*_{\nu(j)}e^{A^*(t_k-t_j)}z\|_{(L^2(\Omega))^m}+\delta\|z\|_{(L^2(\Omega))^n},
\end{equation}
then we call $[A^*,\mathcal{B}^*_{\hbar},\Lambda_{\hbar}]$  $\delta$-approximate observable in $[0,t_k]$,
$D(k,\delta)$ is called a $\delta$-observability constant in $[0,t_{k}]$, and
\begin{equation*}
D_{op}(k,\delta):=\inf\{D(k,\delta):
D(k,\delta)\;\;\mbox{satisfies}\;\;(\ref{Preli-3:1})\}
\end{equation*}
is called the optimal $\delta$-observability constant in $[0,t_k]$.

\end{definition}

\begin{remark}\label{Preli-4}
Let $\Lambda_\hbar\in \mathcal{M}_\hbar$ and $k\in \mathbb{N}^+$ be fixed.
\begin{itemize}
\item[(i)]  Let $\delta>0$. If $[A^*,\mathcal{B}_\hbar^*,\Lambda_\hbar]$ is $\delta$-approximate observable in
$[0,t_k]$, then for all $z\in (L^2(\Omega))^n$,
\begin{equation*}
\|e^{A^*t_{k}}z\|_{(L^2(\Omega))^n}
\leq D_{op}(k,\delta)\sum_{j=1}^{k}
\|B^*_{\nu(j)}e^{A^*(t_k-t_j)}z\|_{(L^2(\Omega))^m}+\delta\|z\|_{(L^2(\Omega))^n}.
\end{equation*}
\item[(ii)] Let $0<\delta_1<\delta_2<+\infty$. If $[A^*,\mathcal{B}_\hbar^*,\Lambda_\hbar]$ is $\delta_1$-approximate observable in
$[0,t_k]$, then it is also $\delta_2$-approximate observable in $[0,t_k]$ and
\begin{equation*}
D_{op}(k,\delta_2)\leq D_{op}(k,\delta_1).
\end{equation*}
\end{itemize}

\end{remark}

\begin{proposition}\label{Preli-5}
Let $\Lambda_\hbar\in \mathcal{M}_\hbar$,
$\delta>0$ and $\gamma\in \mathbb{N}^+$. Suppose that $[A^*,\mathcal{B}_{\hbar}^*,\Lambda_{\hbar}]$ is
$\delta$-approximate observable in $[0,t_{\gamma\hbar}]$. Then for each $k\in\mathbb{N}^+$,
$[A^*,\mathcal{B}_{\hbar}^*,\Lambda_{\hbar}]$ is $\delta_k$-approximate observable in
$[0,t_{k\gamma\hbar}]$ with observability constant $D(k\gamma\hbar,\delta_k)$, where
\begin{equation}\label{Preli-5:1}
\delta_k:=\delta\left(\displaystyle{\sum_{i=0}^{k-1}}
\|e^{A^*t_{i\gamma\hbar}}\|_{{\mathcal L}((L^2(\Omega))^n;(L^2(\Omega))^n)}\right)\Big/
    \left(\displaystyle{\sum_{i=0}^{k-1}}\|e^{A^*t_{i\gamma\hbar}}\|^{-1}_{{\mathcal L}((L^2(\Omega))^n;(L^2(\Omega))^n)}\right),
\end{equation}
and
\begin{equation}\label{Preli-5:2}
D(k\gamma\hbar,\delta_k):=D_{op}(\gamma\hbar,\delta)\Big/
   \left(\displaystyle{\sum_{i=0}^{k-1}}\|e^{A^*t_{i\gamma\hbar}}\|_{{\mathcal L}((L^2(\Omega))^n;(L^2(\Omega))^n)}^{-1}\right).
\end{equation}

\end{proposition}
\begin{proof} Since $[A^*,\mathcal{B}_{\hbar}^*,\Lambda_{\hbar}]$
is $\delta$-approximate observable in $[0,t_{\gamma\hbar}]$,
by $(i)$ in Remark~\ref{Preli-4}, we have that for each $z\in (L^2(\Omega))^n$,
\begin{equation}\label{Preli-5:3}
\|e^{A^*t_{\gamma\hbar}}z\|_{(L^2(\Omega))^n}
\leq D_{op}(\gamma\hbar,\delta)\sum_{j=1}^{\gamma\hbar}
\|B^*_{\nu(j)}e^{A^*(t_{\gamma\hbar}-t_j)}z\|_{(L^2(\Omega))^m}+\delta\|z\|_{(L^2(\Omega))^n}.
\end{equation}
We arbitrarily fix $k\in\mathbb{N}^+$. On one hand, for each $1\leq i\leq k$,
it follows from (\ref{Preli-5:3}), (\ref{Intro-2}) and (\ref{Intro-4}) that
\begin{eqnarray}\label{Preli-5:4}
&\;&\|e^{A^*t_{(k-i+1)\gamma\hbar}}z\|_{(L^2(\Omega))^n}\nonumber\\
&\leq& D_{op}(\gamma\hbar,\delta)\displaystyle{\sum_{j=1}^{\gamma\hbar}}\|B^*_{\nu(j)}
    e^{A^*(t_{(k-i+1)\gamma\hbar}-t_j)}z\|_{(L^2(\Omega))^m}
    +\delta\|e^{A^*(t_{(k-i+1)\gamma\hbar}-t_{\gamma\hbar})}z\|_{(L^2(\Omega))^n}
    \nonumber\\
    &=&D_{op}(\gamma\hbar,\delta)\displaystyle{\sum_{j=1}^{\gamma\hbar}}\|B^*_{\nu(j)}
    e^{A^*(t_{k\gamma\hbar}-t_{(i-1)\gamma\hbar+j})}z\|_{(L^2(\Omega))^m}
    +\delta\|e^{A^*t_{(k-i)\gamma\hbar}}z\|_{(L^2(\Omega))^n}\nonumber\\
    &=&D_{op}(\gamma\hbar,\delta)
    \displaystyle{\sum_{j=(i-1)\gamma\hbar+1}^{i\gamma\hbar}}\|B^*_{\nu(j)}
    e^{A^*(t_{k\gamma\hbar}-t_j)}z\|_{(L^2(\Omega))^m}
    +\delta\|e^{A^*t_{(k-i)\gamma\hbar}}z\|_{(L^2(\Omega))^n}.
\end{eqnarray}
On the other hand, for each $1\leq i\leq k$, by (\ref{Intro-2}), we have that
\begin{eqnarray*}
\|e^{A^*t_{k\gamma\hbar}}z\|_{(L^2(\Omega))^n}&=&
    \|e^{A^*(t_{(i-1)\gamma\hbar}+t_{(k-i+1)\gamma\hbar})}z\|_{(L^2(\Omega))^n}\\
&\leq&\|e^{A^*t_{(i-1)\gamma\hbar}}\|_{{\mathcal L}((L^2(\Omega))^n;(L^2(\Omega))^n)}
    \|e^{A^*t_{(k-i+1)\gamma\hbar}}z\|_{(L^2(\Omega))^n}
\end{eqnarray*}
and
\begin{equation*}
    \|e^{A^*t_{(k-i)\gamma\hbar}}z\|_{(L^2(\Omega))^n}\leq
    \|e^{A^*t_{(k-i)\gamma\hbar}}\|_{{\mathcal L}((L^2(\Omega))^n;(L^2(\Omega))^n)}
    \|z\|_{(L^2(\Omega))^n}.
\end{equation*}
These, along with (\ref{Preli-5:4}), imply that for each $1\leq i\leq k$,
\begin{equation*}\label{Preli-5:5}
\begin{array}{lll}
&&\|e^{A^*t_{(i-1)\gamma\hbar}}\|^{-1}_{{\mathcal L}((L^2(\Omega))^n;(L^2(\Omega))^n)}\|e^{A^*t_{k\gamma\hbar}}z\|_{(L^2(\Omega))^n}\\
&\leq& D_{op}(\gamma\hbar,\delta)\displaystyle{\sum_{j=(i-1)\gamma\hbar+1}^{i\gamma\hbar}}
\|B^*_{\nu(j)}e^{A^*(t_{k\gamma\hbar}-t_j)}z\|_{(L^2(\Omega))^m}
+\delta\|e^{A^*t_{(k-i)\gamma\hbar}}\|_{{\mathcal L}((L^2(\Omega))^n;(L^2(\Omega))^n)}\|z\|_{(L^2(\Omega))^n}.
\end{array}
\end{equation*}
Summing the above inequality with respect to $i$ from $1$ to $k$, we obtain that
\begin{eqnarray*}
&\;&\left(\sum_{i=1}^{k}\|e^{A^*t_{(i-1)\gamma\hbar}}\|_{{\mathcal L}((L^2(\Omega))^n;(L^2(\Omega))^n)}^{-1}\right)
\|e^{A^*t_{k\gamma\hbar}}z\|_{(L^2(\Omega))^n}\\
&\leq&D_{op}(\gamma\hbar,\delta)\sum_{i=1}^{k\gamma\hbar}\|B^*_{\nu(i)}e^{A^*(t_{k\gamma\hbar}-t_i)}z\|_{(L^2(\Omega))^m}  +\delta\left(\sum_{i=1}^k\|e^{A^*t_{(k-i)\gamma\hbar}}\|_{{\mathcal L}((L^2(\Omega))^n;(L^2(\Omega))^n)}\right)\|z\|_{(L^2(\Omega))^n}.
\end{eqnarray*}
Hence, the result follows from the latter inequality immediately.
\end{proof}

At the end of this section, we introduce the following notations and result.
Denote

\begin{equation}\label{Theorem-18}
    \mathcal{F}(k):=\Big\{\eta\in\mathbb{R}^n: \eta=\sum_{j=1}^k e^{(\lambda_1 I_n-P)t_j}
    Q_{\nu(j)}\xi_j,\;(\xi_j)_{j\in\mathbb{N}^+}\in {\mathcal{V}}\Big\}\;\mbox{for each}\,\,k\in \mathbb N^+£¬
\end{equation}
and
\begin{equation}\label{Theorem-21}
    \mathcal{F}:=\bigcup_{k\in\mathbb{N}^+}\mathcal{F}(k),
\end{equation}
where
\begin{equation}\label{yu-5-3-1}
   {\mathcal{V}}:=\{(\xi_j)_{j\in \mathbb{N}^+}\in l^\infty(\mathbb{N}^+;\mathbb{R}^m):
    \|\xi_j\|_{\mathbb{R}^m}\leq 1\;\;\mbox{for all}\;\;j\in\mathbb{N}^+\}.
\end{equation}
Then we have
\begin{proposition}\label{PRO-1}
Let $\Lambda_{\hbar}\in\mathcal{M}_{\hbar}$ be fixed.
   Assume that there is a $k^*\in \mathbb{N}^+$ so that
   \begin{equation}\label{20-2-122}
   \mbox{rank}\;\left( e^{-P t_1}Q_{\nu(1)},
  e^{-Pt_2}Q_{\nu(2)},\cdots,e^{-P t_{k^*}}Q_{\nu(k^*)}\right)=n.
  \end{equation}
If $\sigma(P)\subset \{\rho\in\mathbb{C}:\mbox{Re}(\rho) \leq \lambda_1\}$, then
 \begin{equation}\label{Theorem-26}
    \mathcal{F}=\mathbb{R}^n.
\end{equation}
\end{proposition}

\begin{proof}
The  proof will be organized by three steps.
\vskip 5pt

\emph{Step 1. We show that
there is a constant $C(k^*)>0$ so that
\begin{equation}\label{Theorem-23}
    \|\xi\|_{\mathbb{R}^n}^2\leq C(k^*)\sum_{j=1}^{k^*}\|Q^\top_{\nu(j)}e^{(\lambda_1 I_n-P^\top)t_j}
    \xi\|^2_{\mathbb{R}^m}
    \;\;\mbox{for all}\;\;\xi\in\mathbb{R}^n.
\end{equation}}
\par
Indeed, since
\begin{equation*}
e^{(\lambda_1I_n-P)t_j}Q_{\nu(j)}=e^{\lambda_1 t_j} e^{-Pt_j}Q_{\nu(j)} \textrm{ for each  } 1\leq j\leq k^*,
\end{equation*}
we have that
\begin{equation}\label{20-2-25}
\textrm{span}\{(e^{\lambda_1I_n-P)t_1}Q_{\nu(1)},\cdots, e^{(\lambda_1I_n-P)t_{k^*}}Q_{\nu(k^*)}\}
=\textrm{span}\{e^{-Pt_1}Q_{\nu(1)},\cdots, e^{-Pt_{k^*}}Q_{\nu(k^*)}\}.
\end{equation}
Here and throughout this paper,  span$\{Q\}$ denotes the linear space generated by the columns of the matrix $Q$.
It follows from \eqref{20-2-25} and \eqref{20-2-122}  that
 \begin{equation*}
   \mbox{rank}\;\left( e^{(\lambda_1I_n-P) t_1}Q_{\nu(1)},
  e^{(\lambda_1I_n-P) t_2}Q_{\nu(2)},\cdots,e^{(\lambda_1I_n-P)  t_{k^*}}Q_{\nu(k^*)}\right)=n.
  \end{equation*}
This, along with Lemma \ref{Preli-11} (where  $k$,  $\{\tau_j\}_{j=1}^k$, $\widetilde P$ and  $\{\widetilde Q_j\}_{j=1}^k$
are replaced by $k^*$, $\{t_j\}_{j=1}^{k^*}$, $-\lambda_1I_n+P$ and  $\{Q_{\nu(j)}\}_{j=1}^{k^*}$, respectively),
 yields  \eqref{Theorem-23}.

\vskip 5pt

\emph{Step 2. We claim that there is a constant $\delta>0$ so that
\begin{equation}\label{Theorem-22}
    B^n_{\delta}(0)\subset\mathcal{F}(k^*),
\end{equation}
where $B_\delta^n(0):=\{\eta\in\mathbb{R}^n: \|\eta\|_{\mathbb{R}^n}\leq \delta\}$.}

According to  (\ref{Theorem-23}), the matrix
\begin{equation*}
    M:=\sum_{j=1}^{k^*}e^{(\lambda_1I_n-P)t_j} Q_{\nu(j)}Q_{\nu(j)}^\top e^{(\lambda_1 I_n-P^\top)t_j}
\end{equation*}
is positive definite. Let
\begin{equation}\label{Theorem-24}
    \delta:=\|(e^{(\lambda_1 I_n-P)t_1}Q_{\nu(1)},
    \cdots,e^{(\lambda_1I_n-P)t_{k^*}}
    Q_{\nu(k^*)})\|^{-1}_{\mathbb{R}^{n\times m k^*}}\|M^{-1}\|^{-1}_{\mathbb{R}^{n\times n}}.
\end{equation}
Now, for each $\eta\in B_{\delta}^n(0)$, we set
\begin{equation}\label{Theorem-25}
\zeta_j:=Q_{\nu(j)}^\top e^{(\lambda_1 I_n-P^\top)t_j}M^{-1}\eta\;\;\mbox{for each}\;\;1\leq j\leq k^*
\end{equation}
and
\begin{equation*}
\zeta:=(\zeta_1,\zeta_2,\cdots,\zeta_{k^*},0,0,\cdots).
\end{equation*}
Since $\|\eta\|_{\mathbb{R}^n}\leq \delta$, by (\ref{Theorem-24}), we have that
$\|\zeta_j\|_{\mathbb{R}^m}\leq 1$ for each $1\leq j\leq k^*$. This yields that
$\zeta\in {\mathcal{V}}$.  Moreover, by (\ref{Theorem-25}), we get that
\begin{equation*}
    \sum_{j=1}^{k^*}e^{(\lambda_1 I_n-P)t_j}Q_{\nu(j)}\zeta_j
    =\sum_{j=1}^{k^*}e^{(\lambda_1 I_n-P)t_j}Q_{\nu(j)}Q^\top_{\nu(j)}e^{(\lambda_1 I_n-P^\top)t_j}
    M^{-1}\eta=\eta.
\end{equation*}
Hence, $\eta\in\mathcal{F}(k^*)$ and (\ref{Theorem-22}) follows.

\vskip 5pt

\emph{Step 3. We prove \eqref{Theorem-26}.}

By contradiction, there would exist $\xi\in \mathbb{R}^n\setminus {\mathcal F}$.
According to (\ref{Theorem-18}),
\begin{equation}\label{Theorem-26-1}
{\mathcal F}(k)\; (\forall\; k\in \mathbb{N}^+)\;\;\mbox{is convex and}\;\;
{\mathcal F}(k_1)\subset {\mathcal F}(k_2)\;\;\mbox{if}\;\;k_1\leq k_2.
\end{equation}
From the latter, (\ref{Theorem-21}) and (\ref{Theorem-22}), it follows that ${\mathcal F}$
is convex and has a nonempty interior. By the Hahn-Banach theorem,
there is a $\phi\in \mathbb{R}^n\setminus \{0\}$ so that
\begin{equation}\label{Theorem-27}
    \langle f,\phi\rangle_{\mathbb{R}^n}\leq \langle \xi,\phi\rangle_{\mathbb{R}^n}\;\;\mbox{for all}\;\;f\in\mathcal{F}.
\end{equation}
In what follows, we will construct a sequence $\{f_\ell\}_{\ell\in\mathbb{N}^+}\subset\mathcal{F}$ so that
\begin{equation}\label{Theorem-28}
    \langle f_\ell,\phi\rangle_{\mathbb{R}^n}\to +\infty\;\;\mbox{as}\;\;\ell\rightarrow +\infty.
\end{equation}
When (\ref{Theorem-28}) is proved, by (\ref{Theorem-27}), we will arrive at a contradiction.
Then (\ref{Theorem-26}) follows.

For this purpose, firstly, we choose a $\gamma\in \mathbb{N}^+$ so that
$\widehat{k}:=\gamma \hbar\geq k^*$. By (\ref{Theorem-22}) and the second claim in (\ref{Theorem-26-1}),
we have that
\begin{equation}\label{Theorem-29}
B_\delta^n(0)\subset {\mathcal F}(k)\;\;\mbox{for any}\;\;k\geq \widehat k.
\end{equation}
Here $\delta>0$ is defined by (\ref{Theorem-24}). Let
\begin{equation}\label{Theorem-30}
z_i:=\delta\frac{e^{(\lambda_1 I_n-P^\top)t_{i\widehat{k}}}\phi}
{\|e^{(\lambda_1 I_n-P^\top)t_{i\widehat{k}}}\phi\|_{\mathbb{R}^n}}
\;\;\mbox{for each}\;\;i\in \mathbb{N}.
\end{equation}
This, along with (\ref{Theorem-29}), implies that
\begin{equation*}
z_i\in B_\delta^n(0)\subset {\mathcal F}(\widehat{k})\;\;
\mbox{for each}\;\;i\in \mathbb{N}.
\end{equation*}
It follows from the latter and (\ref{Theorem-18}) that for each $i\in \mathbb{N}$,
there is a sequence $(\xi_j^{(i)})_{j\in\mathbb{N}^+}\in {\mathcal{V}}$ so that
\begin{equation}\label{Theorem-31}
    z_i=\sum_{j=1}^{\widehat{k}} e^{(\lambda_1 I_n-P)t_j}Q_{\nu(j)}\xi_j^{(i)}.
\end{equation}

Secondly, we set
\begin{equation}\label{}\label{Theorem-32-1}
    f_{\ell}:=z_0+e^{(\lambda_1 I_n-P)t_{\widehat{k}}}z_1+\cdots+
    e^{(\lambda_1 I_n-P)t_{\ell \widehat{k}}}z_{\ell}\;\;\mbox{for each}\;\;\ell\in\mathbb{N}^+.
\end{equation}
Since $\Lambda_\hbar=\{t_j\}_{j\in \mathbb{N}}\in\mathcal M_\hbar$, by (\ref{Theorem-31}), (\ref{Theorem-32-1}),
(\ref{Intro-2}) and (\ref{Intro-4}), we have that
\begin{eqnarray*}
    f_\ell&=&\sum_{j=1}^{\widehat{k}}e^{(\lambda_1 I_n-P)t_j}Q_{\nu(j)}\xi_j^{(0)}
    +e^{(\lambda_1 I_n-P)t_{\widehat{k}}}\sum_{j=1}^{\widehat{k}}e^{(\lambda_1 I_n-P)t_j}Q_{\nu(j)}\xi_j^{(1)}\\
    &&+\cdots+e^{(\lambda_1 I_n-P)t_{\ell\widehat{k}}}\sum_{j=1}^{\widehat{k}}  e^{(\lambda_1 I_n-P)t_j}
    Q_{\nu(j)}\xi_j^{(\ell)}\\
    &=&\sum_{j=1}^{\widehat{k}}e^{(\lambda_1 I_n-P)t_j}Q_{\nu(j)}\xi_j^{(0)}
    +\sum_{j=1}^{\widehat{k}}e^{(\lambda_1 I_n-P)t_{\widehat{k}+j}}Q_{\nu(\widehat{k}+j)}\xi_j^{(1)}
    +\cdots+\sum_{j=1}^{\widehat{k}}e^{(\lambda_1 I_n-P)t_{\ell\widehat{k}+j}}Q_{\nu(\ell\widehat{k}+j)}
    \xi_j^{(\ell)}\\
    &=&\sum_{j=1}^{\widehat{k}}e^{(\lambda_1 I_n-P)t_j}Q_{\nu(j)}\xi_j^{(0)}
    +\sum_{j=\widehat{k}+1}^{2\widehat{k}} e^{(\lambda_1 I_n-P)t_j}Q_{\nu(j)}\xi_{j-\widehat{k}}^{(1)}
    +\cdots+\sum_{j=\ell\widehat{k}+1}^{(\ell+1)\widehat{k}}e^{(\lambda_1 I_n-P)t_j}Q_{\nu(j)}\xi_{j-\ell
    \widehat{k}}^{(\ell)}.
\end{eqnarray*}
This, along with the fact that $(\xi_j^{(i)})_{j\in\mathbb{N}^+}\in{\mathcal{V}}$ (for each
$i\in\mathbb{N}$), (\ref{Theorem-18}) and (\ref{Theorem-21}),  implies that
\begin{equation*}
    f_\ell\in\mathcal{F}((\ell+1)\widehat{k})\subset\mathcal{F}\;\;\mbox{for each}\;\;
    \ell\in\mathbb{N}^+.
\end{equation*}

Finally, we check that $\{f_\ell\}_{\ell\in\mathbb{N}^+}$ defined by (\ref{Theorem-32-1})
satisfies (\ref{Theorem-28}). According to (\ref{Theorem-32-1}) and (\ref{Theorem-30}),
for each $\ell\in\mathbb{N}^+$,
\begin{equation*}
    \langle f_\ell,\phi\rangle_{\mathbb{R}^n}
    =\delta\left(\|\phi\|_{\mathbb{R}^n}+\|e^{(\lambda_1 I_n-P^\top)t_{\widehat{k}}}\phi\|_{\mathbb{R}^n}
    +\cdots+\|e^{(\lambda_1I_n-P^\top)t_{\ell\widehat{k}}}\phi\|_{\mathbb{R}^n}\right).
\end{equation*}
It is clear that $\{\langle f_{\ell},\phi\rangle_{\mathbb{R}^n}\}_{\ell\in \mathbb{N}^+}$
is a strictly monotonically increasing sequence.
We will use a contradiction argument to show (\ref{Theorem-28}). Otherwise,
\begin{equation}\label{Theorem-33}
    \|e^{(\lambda_1 I_n-P^\top)t_{\ell\widehat{k}}}\phi\|_{\mathbb{R}^n}\rightarrow 0\;\;
    \mbox{when}\;\;\ell\rightarrow +\infty.
\end{equation}
On one hand, for any $t\geq \widehat{k}$, there exists a unique $\ell\in\mathbb{N}^+$
so that $t\in [t_{\ell\widehat{k}},t_{(\ell+1)\widehat{k}})$. Then
\begin{equation}\label{Theorem-34}
    \|e^{(\lambda_1 I_n-P^\top) t}\phi\|_{\mathbb{R}^n}
    =\|e^{(\lambda_1 I_n-P^\top)(t-t_{\ell\widehat{k}})}
    e^{(\lambda_1 I_n-P^\top) t_{\ell\widehat{k}}}\phi\|_{\mathbb{R}^n}.
\end{equation}
By (\ref{Intro-2}), we have that
\begin{equation*}
0\leq t-t_{\ell\widehat{k}}\leq t_{(\ell+1)\widehat{k}}-t_{\ell\widehat{k}}
=t_{\widehat{k}}=t_{\gamma\hbar},
\end{equation*}
which, combined with (\ref{Theorem-34}), indicates that
\begin{equation}\label{Theorem-35}
   \|e^{(\lambda_1 I_n-P^\top)t}\phi\|_{\mathbb{R}^n}\leq C(\gamma\hbar)
   \|e^{(\lambda_1 I_n-P^\top)t_{\ell\widehat{k}}}\phi\|_{\mathbb{R}^n}.
\end{equation}
Here $C(\gamma\hbar)>0$ is a constant independent of $\ell$. It follows from (\ref{Theorem-33})
and (\ref{Theorem-35}) that
\begin{equation}\label{Theorem-36}
\|e^{(\lambda_1 I_n-P^\top)t}\phi\|_{\mathbb{R}^n}\rightarrow 0\;\;\mbox{when}\;\;t\rightarrow +\infty.
\end{equation}
On the other hand, we denote $\sigma(\lambda_1 I_n-P^\top)=\{\mu_k\}_{k=1}^s$ ($s\in\mathbb{N}^+$).
Since $\sigma(P)\subset \{\rho\in\mathbb{C}:\mbox{Re}(\rho)\leq \lambda_1\}$,
we have that
\begin{equation*}
    \mbox{Re}(\mu_k)\geq 0\;\;\mbox{for each}\;\;1\leq k\leq s.
\end{equation*}
which, combined with (\ref{Theorem-36}), indicates that $\phi=0$. It contradicts to the fact
that $\phi\not=0$. Thus (\ref{Theorem-28}) follows.
\par
In summary, we finish the proof of Proposition \ref{PRO-1}.
\end{proof}

\section{Proof  of Theorem~\ref{Intro-6}}

The proof of Theorem~\ref{Intro-6} will be organized by  two subsections.

\subsection{Proof of $(i)$}

We divide its proof into the following four steps.
\vskip 5pt

\emph{Step 1. We decompose the state space $(L^2(\Omega))^n$.}

Recall that $e_1$ is the eigenfunction of $-\triangle$ with respect to $\lambda_1$. Define
\begin{equation}\label{Theorem-1}
    H_{1}:=\mathbb{R}^n e_1=\{v e_1: v\in \mathbb{R}^n\},
\end{equation}
the projection operator
\begin{equation}\label{Theorem-2}
\mathcal{P}: (L^2(\Omega))^n\rightarrow H_1
\end{equation}
and
\begin{equation}\label{Theorem-3}
H_1^\bot:=(I-\mathcal{P})(L^2(\Omega))^n.
\end{equation}
By (\ref{Theorem-1})-(\ref{Theorem-3}), we have that
\begin{equation}\label{Theorem-4}
H_1\;\;\mbox{is a finite dimensional space and}\;\;
(L^2(\Omega))^n=H_1\oplus H_1^\bot.
\end{equation}

\vskip 5pt

\emph{Step 2. We show that for each $f\in H_1$, there is a $k(f)\in \mathbb{N}^+$ and
a sequence $(u_j)_{j\in \mathbb{N}^+}:=(\xi_j e_1)_{j\in \mathbb{N}^+}\in \mathcal{U}$
 so that
\begin{equation}\label{Theorem-14}
e^{At_{k(f)}}f+\sum_{j=1}^{k(f)}e^{A(t_{k(f)}-t_j)} B_{\nu(j)} u_j=0.
\end{equation}}


 Indeed, by (\ref{Theorem-1}) and (\ref{Theorem-26}), we obtain that
\begin{equation*}
    \mathcal{F}e_1=H_1.
\end{equation*}
This, along with (\ref{Theorem-21}) and (\ref{Theorem-18}), implies that
for each $f\in H_1$, there is a $k(f)\in \mathbb{N}^+$ and a sequence
$(\xi_j)_{j\in\mathbb{N}^+}\in \mathcal{V}$ (Recall (\ref{yu-5-3-1}) for the definition of $\mathcal{V}$) so that
\begin{equation*}
f+\sum_{j=1}^{k(f)}e^{(\lambda_1 I_n-P)t_j}Q_{\nu(j)}\xi_j e_1=0,
\end{equation*}
which indicates that
\begin{equation}\label{Theorem-37}
0=f+\sum_{j=1}^{k(f)}e^{-At_j}|_{H_1}Q_{\nu(j)}\xi_j e_1.
\end{equation}
Let $(u_j)_{j\in \mathbb{N}^+}:=(\xi_j e_1)_{j\in \mathbb{N}^+}$. It is clear that
$(u_j)_{j\in \mathbb{N}^+}\in \mathcal{U}$ and it follows from (\ref{Theorem-37}) that
\begin{equation*}
0=e^{A t_{k(f)}}f+\sum_{j=1}^{k(f)} e^{A(t_{k(f)}-t_j)}Q_{\nu(j)}u_j=
e^{A t_{k(f)}}f+\sum_{j=1}^{k(f)} e^{A(t_{k(f)}-t_j)}B_{\nu(j)}u_j.
\end{equation*}
Hence, (\ref{Theorem-14}) follows.

\vskip 5pt

\emph{Step 3.
We show that there is a constant $C>0$ so that
\begin{equation}\label{Theorem-9}
\|e^{At}g\|_{(L^2(\Omega))^n}\leq Ce^{-(\lambda_2-\lambda_1)t/2}\|g\|_{(L^2(\Omega))^n}
\;\;\mbox{for all}\;\;g\in H_1^\bot.
\end{equation}}
\par
To achieve this goal,  on one hand, since $\sigma(P)\subset \{\rho\in \mathbb{C}: \mbox{Re}\rho\leq \lambda_1\}$
and $0<\lambda_1<\lambda_2$, there is a constant $C>0$ so that
\begin{equation}\label{Theorem-7}
\|e^{Pt}\eta\|_{\mathbb{R}^n}\leq C e^{(\lambda_1+\lambda_2)t/2}\|\eta\|_{\mathbb{R}^n}
\;\;\mbox{for all}\;\;\eta\in \mathbb{R}^n\;\;\mbox{and}\;\;t\in \mathbb{R}^+.
\end{equation}
On the other hand, for each $g\in H_1^\bot$, there exists a sequence $\{g_i\}_{i\geq 2}\subset \mathbb{R}^n$
so that
\begin{equation}\label{Theorem-8}
g=\sum_{i=2}^{+\infty} g_i e_i\;\;\mbox{and}\;\;
\|g\|_{(L^2(\Omega))^n}^2=\sum_{i=2}^{+\infty} \|g_i\|_{\mathbb{R}^n}^2.
\end{equation}
Thus
\begin{equation*}
e^{At} g=e^{Pt}\sum_{i=2}^{+\infty} e^{-\lambda_i t}g_i e_i
=\sum_{i=2}^{+\infty} e^{-\lambda_i t}e_i e^{Pt}g_i,
\end{equation*}
which indicates that
\begin{equation*}
\|e^{At}g\|_{(L^2(\Omega))^n}^2
=\sum_{i=2}^{+\infty} e^{-2\lambda_i t}\|e^{Pt}g_i\|_{\mathbb{R}^n}^2.
\end{equation*}
It follows from the latter, (\ref{Theorem-7}) and (\ref{Theorem-8}) that
\begin{equation*}
\|e^{At}g\|_{(L^2(\Omega))^n}^2\leq C\sum_{i=2}^{+\infty}
e^{-2\lambda_i t}e^{(\lambda_2+\lambda_1)t}\|g_i\|_{\mathbb{R}^n}^2
\leq C e^{-(\lambda_2-\lambda_1)t}\|g\|^2_{(L^2(\Omega))^n}.
\end{equation*}
Hence, (\ref{Theorem-9}) follows immediately.

\vskip 5pt

\emph{Step 4. We show that for each $\varepsilon>0$ and $x_0\in (L^2(\Omega))^n\setminus \{0\}$,
there is a $k\in\mathbb{N}^+$ and a control sequence $(u_j)_{j\in\mathbb{N}^+}\in\mathcal{U}$ so that
\begin{equation}\label{Theorem-38}
    e^{A t_k}x_0+\sum_{j=1}^k e^{A(t_k-t_j)}B_{\nu(j)}u_j\in B_{\varepsilon}(0).
\end{equation}}

Indeed, according to (\ref{Theorem-4}), there are
$x_{0,1}\in H_1$ and $x_{0,2}\in\ H_1^\bot$ so that
\begin{equation}\label{Theorem-39}
    x_0=x_{0,1}+x_{0,2}.
\end{equation}
On one hand, by (\ref{Theorem-14}), there is a $k_0\in \mathbb{N}^+$ and
a sequence $(v_j)_{j\in\mathbb{N}^+}\in \mathcal{U}$ so that
\begin{equation}\label{Theorem-40}
    e^{At_{k_0}}x_{0,1}+\sum_{j=1}^{k_0}e^{A(t_{k_0}-t_j)}B_{\nu(j)}v_j=0.
\end{equation}
On the other hand, by (\ref{Theorem-9}), there is a $k\geq k_0$ so that
\begin{equation}\label{Theorem-41}
    e^{A t_k}x_{0,2}\in B_\varepsilon(0).
\end{equation}
Define
\begin{equation}\label{Theorem-42}
    u_j:=
\begin{cases}
    v_j&\mbox{if}\;\;1\leq j\leq k_0,\\
    0&\mbox{if}\;\;j\geq k_0+1.
\end{cases}
\end{equation}
Then $(u_j)_{j\in \mathbb{N}^+}\in\mathcal{U}$ and it follows from
(\ref{Theorem-39})-(\ref{Theorem-42}) that
\begin{eqnarray*}
    e^{A t_k}x_0+\sum_{j=1}^k e^{A(t_k-t_j)}B_{\nu(j)}u_j
    &=&e^{A(t_k-t_{k_0})}\Big(e^{A t_{k_0}}x_{0,1}
    +\sum_{j=1}^{k_0}e^{A(t_{k_0}-t_j)}B_{\nu(j)}v_j\Big)
    +e^{A t_k}x_{0,2}\\
    &=&e^{A t_k}x_{0,2}\in B_\varepsilon(0).
\end{eqnarray*}
Hence, (\ref{Theorem-38}) holds.
\par
In summary, we finish the proof of $(i)$.

\subsection{Proof of $(ii)$}
We divide its proof into the following three steps.

\vskip 5pt

\emph{Step 1. We show that
  \begin{equation}\label{Theo-1}
    \|e^{At}\|_{{\mathcal L}((L^2(\Omega))^n;(L^2(\Omega))^n)}\leq 1
    \;\;\mbox{for all}\;\;t\in\mathbb{R}^+.
\end{equation}}

For this purpose, on one hand, since
\begin{equation}\label{Theo-2}
\langle P\eta,\eta\rangle_{\mathbb{R}^n}\leq \lambda_1\|\eta\|_{\mathbb{R}^n}^2
\;\;\mbox{for each}\;\;\eta\in \mathbb{R}^n,
\end{equation}
we have that
\begin{equation}\label{Theo-3}
\begin{array}{lll}
    \langle Ax,x\rangle_{(L^2(\Omega))^n}&=&\langle (\triangle_n+P)x,x\rangle_{(L^2(\Omega))^n}\\
    &\leq&-\lambda_1\|x\|_{(L^2(\Omega))^n}^2+\langle Px,x\rangle_{(L^2(\Omega))^n}\leq 0
    \;\;\mbox{for all}\;\;x\in D(A).
\end{array}
\end{equation}
On the other hand, we fix any positive constant
$\rho_0>\lambda_1$. For any $x=(x_1,x_2,\cdots,x_n)^\top\in (H_0^1(\Omega))^n
$ and $y=(y_1,y_2,\cdots,y_n)^\top\in (H_0^1(\Omega))^n$, we define
\begin{equation*}
    \pi(x,y):=\rho_0\langle x,y\rangle_{(L^2(\Omega))^n}
    -\langle Px,y\rangle_{(L^2(\Omega))^n}
    +\sum_{i=1}^n\langle \nabla x_i,\nabla y_i\rangle_{(L^2(\Omega))^n}.
\end{equation*}
It is clear that there is a constant $C>0$ so that
\begin{equation}\label{Theo-4}
    |\pi(x,y)|\leq C\|x\|_{(H_0^1(\Omega))^n}\|y\|_{(H_0^1(\Omega))^n}
    \;\;\mbox{for all}\;\;x, y\in (H_0^1(\Omega))^n.
\end{equation}
Moreover, it follows from (\ref{Theo-2}) that
\begin{equation}\label{Theo-5}
    \pi(x,x)\geq\sum_{i=1}^n\langle \nabla x_i,\nabla x_i\rangle_{(L^2(\Omega))^n}
    =\|x\|_{(H_0^1(\Omega))^n}^2\;\;\mbox{for all}\;\;x\in (H_0^1(\Omega))^n.
\end{equation}
According to (\ref{Theo-4}), (\ref{Theo-5}) and the Lax-Milgram theorem,
for each $f\in (H^{-1}(\Omega))^n$, there is a unique $x_f\in (H_0^1(\Omega))^n$ so that
\begin{equation*}
    \pi(x_f,y)=\langle f,y\rangle_{(H^{-1}(\Omega))^n,(H_0^1(\Omega))^n}
    \;\;\mbox{for all}\;\;y\in (H_0^1(\Omega))^n,
\end{equation*}
i.e.,
\begin{equation}\label{Theo-6}
    (\rho_0 I-A)x_f=f\;\;\mbox{in}\;\;(H^{-1}(\Omega))^n,
\end{equation}
where $I$ is the identity operator in ${\mathcal L}((L^2(\Omega))^n; (L^2(\Omega))^n)$.
In particular, if $f\in (L^2(\Omega))^n$, by (\ref{Theo-6}) and the definition of $D(A)$,
we have that $x_f\in D(A)$. Hence,
\begin{equation}\label{Theo-7}
\mbox{R}(\rho_0 I-A)=(L^2(\Omega))^n.
\end{equation}
Here $\mbox{R}(\rho_0 I-A)$ denotes the range of $\rho_0 I-A$.
It follows from (\ref{Theo-3}), (\ref{Theo-7}) and the Lumer-Phillips theorem (see \cite{Pazy})
that (\ref{Theo-1}) holds.

\vskip 5pt

\emph{Step 2. We claim that for any $\delta\in (0,1)$ and $k\in \mathbb{N}^+$,
$[A^*,{\mathcal B}_\hbar^*,\Lambda_\hbar]$ is $\delta$-approximate observable
in $[0,t_{2kk^*\hbar}]$ and $\displaystyle{\lim_{k\rightarrow +\infty}} D_{op}(2kk^*\hbar,\delta)=0$.}

Since rank$(e^{-P t_1}Q_{\nu(1)}, e^{-P t_2}Q_{\nu(2)}, \cdots, e^{-P t_{k^*}}Q_{\nu(k^*)})=n$,
we have that
\begin{equation}\label{Theo-8}
\mbox{rank}(e^{-P t_1}Q_{\nu(1)}, e^{-P t_2}Q_{\nu(2)}, \cdots,
e^{-P t_{2k^*\hbar-1}}Q_{\nu(2k^*\hbar-1)})=n.
\end{equation}
According to (\ref{Theo-8}) and Proposition~\ref{Preli-2}, there are two constants
$\theta\in (0,1)$  and $C>0$ so that
for all $z\in (L^2(\Omega))^n$,
\begin{equation*}
\|e^{A^*t_{2k^*\hbar}}z\|_{(L^2(\Omega))^n}\leq C\left(\sum_{j=1}^{2k^*\hbar}
    \|B^*_{\nu(j)}e^{A^*(t_{2k^*\hbar}-t_j)}z\|_{(L^2(\Omega))^m}\right)^{1-\theta}
    \|z\|^\theta_{(L^2(\Omega))^n}.
\end{equation*}
Then for any $\delta\in (0,1)$, there exists a constant $C(\delta)>0$ so that
for all  $z\in (L^2(\Omega))^n$,
\begin{equation*}
 \|e^{A^*t_{2k^*\hbar}}z\|_{(L^2(\Omega))^n}\leq C(\delta)\displaystyle{\sum_{j=1}^{2k^*\hbar}}
    \|B^*_{\nu(j)}e^{A^*(t_{2k^*\hbar}-t_j)}z\|_{(L^2(\Omega))^m}+
    \delta\|z\|_{(L^2(\Omega))^n},
\end{equation*}
which indicates that
\begin{equation*}
[A^*,{\mathcal B}_\hbar^*,\Lambda_\hbar]\;\;
\mbox{is}\;\;\delta\mbox{-approximate observable in}\;\;[0,t_{2k^*\hbar}].
\end{equation*}
This, along with Proposition~\ref{Preli-5}, implies that for each $k\in \mathbb{N}^+$,
\begin{equation}\label{Theo-9}
[A^*,{\mathcal B}_\hbar^*,\Lambda_\hbar]\;\;
\mbox{is}\;\;\delta_k\mbox{-approximate observable in}\;\;
[0,t_{2kk^*\hbar}]
\end{equation}
with observability constant $D(2kk^*\hbar,\delta_k)$. Here,
$\delta_k$ and $D(2kk^*\hbar,\delta_k)$ are defined as (\ref{Preli-5:1})
and (\ref{Preli-5:2}) (where $\gamma=2k^*$), respectively.

Furthermore, by (\ref{Theo-1}), we get that
\begin{equation}\label{Theo-10}
\sum_{i=0}^{k-1}\|e^{A^*t_{2ik^*\hbar}}\|_{{\mathcal L}((L^2(\Omega))^n;(L^2(\Omega))^n)}\leq
\sum_{i=0}^{k-1}\|e^{A^*t_{2ik^*\hbar}}\|^{-1}_{{\mathcal L}((L^2(\Omega))^n;(L^2(\Omega))^n)}\;\;
\mbox{for each}\;\;k\in \mathbb{N}^+,
\end{equation}
and
\begin{equation}\label{Theo-11}
\sum_{i=0}^{+\infty}\|e^{A^*t_{2ik^*\hbar}}\|^{-1}_{{\mathcal L}((L^2(\Omega))^n;(L^2(\Omega))^n)}
=+\infty.
\end{equation}
It follows from (\ref{Preli-5:1}), (\ref{Theo-10}) and (\ref{Theo-9}) that for each $k\in \mathbb{N}^+$,
$\delta_k\leq \delta$, $[A^*,{\mathcal B}_\hbar^*,\Lambda_\hbar]$ is $\delta$-approximate observable in
$[0,t_{2kk^*\hbar}]$ and
\begin{equation*}
D_{op}(2kk^*\hbar,\delta)\leq  D(2kk^*\hbar,\delta_k).
\end{equation*}
This, along with (\ref{Preli-5:2}) and (\ref{Theo-11}), implies that
$\displaystyle{\lim_{k\rightarrow +\infty}} D_{op}(2kk^*\hbar,\delta)=0$.

\vskip 5pt

\emph{Step 3. We show that for any $\varepsilon>0$ and $x_0\in (L^2(\Omega))^n$, there
exists a control $u\in {\mathcal U}$ and $k_0\in \mathbb{N}$ so that
\begin{equation}\label{Theo-12}
x(t_{k_0};x_0,u,\Lambda_\hbar)\in B_\varepsilon(0).
\end{equation}}

To achieve this goal, for each $k\in \mathbb{N}$, we firstly define the
reachable set of (\ref{Intro-3}) at $t_k$ with initial state $x_0$ as follows:
\begin{equation*}
{\mathcal R}(x_0,k):=\{x(t_k;x_0,v,\Lambda_\hbar): v\in {\mathcal U}\}.
\end{equation*}
It is clear that ${\mathcal R}(x_0,k)$ is convex and closed.
Two cases may occur: $x_0\in B_\varepsilon(0)$ or $x_0\not\in B_\varepsilon(0)$.
\vskip 5pt
\emph{Case 1. $x_0\in B_\varepsilon(0)$.} In this case, (\ref{Theo-12}) follows immediately with $k_0=0$.
\vskip 5pt
\emph{Case 2. $x_0\not\in B_\varepsilon(0)$.} In this case, we will use a contradiction
argument to prove (\ref{Theo-12}). By contradiction, for each $k\in \mathbb{N}$, we
would have that
\begin{equation*}\label{Theo-13}
\rho_k:=d_{(L^2(\Omega))^n}({\mathcal R}(x_0,k),B_\varepsilon(0))
=d_{(L^2(\Omega))^n}(B_\varepsilon(0)-{\mathcal R}(x_0,k),\{0\})>0,
\end{equation*}
where $d_{(L^2(\Omega))^n}(E,F):=\displaystyle{\inf_{x\in E, y\in F}} \|x-y\|_{(L^2(\Omega))^n}$ and
$E-F:=\{x-y: x\in E, y\in F\}$ for all $E, F\subset (L^2(\Omega))^n$.

On one hand, since $B_\varepsilon(0)-{\mathcal R}(x_0,k)$ is a convex set and
$B_{\rho_k/2}(0)\cap [B_\varepsilon(0)-{\mathcal R}(x_0,k)]=\emptyset$, by the Hahn-Banach theorem,
there is a $x_0^*\in (L^2(\Omega))^n\setminus\{0\}$ so that
\begin{equation}\label{Theo-14}
\langle x_0^*,f\rangle_{(L^2(\Omega))^n}\leq
\langle x_0^*,x_1-x(t_k;x_0,v,\Lambda_\hbar)\rangle_{(L^2(\Omega))^n}
\end{equation}
for any $f\in B_{\rho_k/2}(0), v=(v_j)_{j\in \mathbb{N}^+}\in {\mathcal U}$ and
$x_1\in B_\varepsilon(0)$. Note that (\ref{Theo-14}) can be equivalently rewritten
as follows:
\begin{equation*}
\begin{array}{l}
\langle x_0^*,f\rangle_{(L^2(\Omega))^n}-\langle x_0^*,x_1\rangle_{(L^2(\Omega))^n}\\
+\displaystyle{\sum_{j=1}^k}\langle v_j,B_{\nu(j)}^* e^{A^*(t_k-t_j)}x_0^*\rangle_{(L^2(\Omega))^m}
\leq -\langle x_0^*,e^{A t_k}x_0\rangle_{(L^2(\Omega))^n}.
\end{array}
\end{equation*}
From the latter it follows that
\begin{equation}\label{Theo-15}
(\varepsilon+\rho_k/2)\|x_0^*\|_{(L^2(\Omega))^n}
+\sum_{j=1}^k \|B_{\nu(j)}^* e^{A^*(t_k-t_j)}x_0^*\|_{(L^2(\Omega))^m}
\leq \|x_0\|_{(L^2(\Omega))^n} \|e^{A^* t_k}x_0^*\|_{(L^2(\Omega))^n}.
\end{equation}
On the other hand, according to \emph{Step 2}, there is a $\widetilde{k}\in \mathbb{N}^+$ so that
$[A^*,{\mathcal B}_\hbar^*,\Lambda_\hbar]$ is $\varepsilon \|x_0\|^{-1}_{(L^2(\Omega))^n}$-approximate
observable in $[0,t_{2\widetilde{k} k^*\hbar}]$ and
$D_{op}(2\widetilde{k} k^*\hbar,\varepsilon\|x_0\|^{-1}_{(L^2(\Omega))^n})\leq \|x_0\|_{(L^2(\Omega))^n}^{-1}$.
Hence, by Remark~\ref{Preli-4}, we get that
\begin{eqnarray*}
&&\|e^{A^* t_{2\widetilde{k} k^*\hbar}}x_0^*\|_{(L^2(\Omega))^n}\\
&\leq&D_{op}(2\widetilde{k} k^*\hbar,\varepsilon\|x_0\|^{-1}_{(L^2(\Omega))^n})
\sum_{j=1}^{2\widetilde{k} k^*\hbar} \|B_{\nu(j)}^* e^{A^*(t_{2\widetilde{k} k^*\hbar}-t_j)}x_0^*\|_{(L^2(\Omega))^m}
+\varepsilon\|x_0\|^{-1}_{(L^2(\Omega))^n}\|x_0^*\|_{(L^2(\Omega))^n}\\
&\leq&\|x_0\|^{-1}_{(L^2(\Omega))^n}
\Big(\sum_{j=1}^{2\widetilde{k} k^*\hbar} \|B_{\nu(j)}^* e^{A^*(t_{2\widetilde{k} k^*\hbar}-t_j)}x_0^*\|_{(L^2(\Omega))^m}
+\varepsilon\|x_0^*\|_{(L^2(\Omega))^n}\Big).
\end{eqnarray*}
This, along with (\ref{Theo-15}), implies that $\rho_{2\widetilde{k} k^*\hbar}\|x_0^*\|_{(L^2(\Omega))^n}\leq 0$.
It leads to a contradiction. Hence, (\ref{Theo-12}) follows.
\par
In summary, we finish the proof of $(ii)$.

\section{Proof of Theorem~\ref{20-2-28}}

By contradiction, there would exist an $\varepsilon_0>0$ and a $\Lambda_\hbar\in\mathcal M_\hbar$ so that $[A,{\mathcal B}_\hbar,\Lambda_\hbar]$ is
$(\varepsilon_0$-$\text{GCAC})_{\hbar}$. Let
$\rho:=\rho_1+\mathfrak{i}\rho_2\in \sigma(P^\top)$ with $\rho_1>\lambda_1$, where $\mathfrak{i}$ is the unit element of pure imaginary number.
Then there is a $\eta\in \mathbb{C}^n$ with $\|\eta\|_{\mathbb{C}^n}=1$ so that
$P^\top\eta=\rho\eta$. Denote $\widehat{\eta}:=(\eta-\overline{\eta})/(2\mathfrak{i})$,
where $\overline{\eta}$ is the
conjugate vector of $\eta$. Two cases may occur: $\widehat{\eta}=0$ or $\widehat{\eta}\not=0$.
\vskip 5pt
\emph{Case 1. $\widehat{\eta}=0$.}  In this case,  $\eta\in\mathbb R^n$ and $\rho_2=0$. We can easily check that
\begin{equation}\label{theorem-0}
    e^{P^\top t}\eta=e^{\rho t}\eta=e^{\rho_1 t}\eta\;\;\mbox{for all}\;\;t\in\mathbb{R}^+.
\end{equation}
Since $[A,{\mathcal B}_\hbar,\Lambda_\hbar]$ is $(\varepsilon_0$-$\text{GCAC})_{\hbar}$,
for each $\ell\geq 2\varepsilon_0$, there exists a control
$(u_j^{(\ell)})_{j\in\mathbb{N}^+}\in\mathcal{U}, k_\ell\in\mathbb{N}^+$ and
$f_\ell\in B_{\varepsilon_0}(0)$ so that
\begin{equation*}
     e^{A t_{k_\ell}}(\ell\eta e_1)+\sum_{j=1}^{k_\ell}
     e^{A(t_{k_\ell}-t_j)}B_{\nu(j)}u_j^{(\ell)}=f_\ell.
\end{equation*}
This, along with (\ref{Preli-2:3}) and (\ref{theorem-0}), implies that
\begin{eqnarray*}
&&\langle f_\ell,\eta e_1\rangle_{(L^2(\Omega))^n}\\
&=&\ell \langle e^{\triangle t_{k_\ell}}e_1 e^{P^\top t_{k_\ell}}\eta,\eta e_1\rangle_{(L^2(\Omega))^n}
+\sum_{j=1}^{k_\ell}\langle u_j^{(\ell)},\chi_{\omega_{\nu(j)}}
    Q^\top_{\nu(j)} e^{\triangle(t_{k_\ell}-t_j)}e_1 e^{P^\top(t_{k_\ell}-t_j)}\eta\rangle_{(L^2(\Omega))^m}\\
    &=&\ell e^{(\rho_1-\lambda_1)t_{k_\ell}}+\sum_{j=1}^{k_\ell}\langle u_j^{(\ell)},\chi_{\omega_{\nu(j)}}Q^\top_{\nu(j)} e^{(\rho_1-\lambda_1)(t_{k_\ell}-t_j)}e_1\eta\rangle_{(L^2(\Omega))^m}.
\end{eqnarray*}
It follows from the latter that
\begin{equation}\label{theorem-1}
\ell e^{(\rho_1-\lambda_1)t_{k_\ell}}\leq \max_{1\leq j\leq \hbar}
\|Q_j^\top\|_{\mathbb{R}^{m\times n}}\sum_{j=1}^{k_\ell} e^{(\rho_1-\lambda_1)(t_{k_\ell}-t_j)}+\varepsilon_0.
\end{equation}
Noting that
\begin{equation*}
\sum_{j=1}^{k_\ell} e^{(\rho_1-\lambda_1)(t_{k_\ell}-t_j)}
\leq \sum_{j=1}^{k_\ell} e^{(\rho_1-\lambda_1)(t_{k_\ell}-t_j)} (t_j-t_{j-1})
\left[\min_{1\leq j\leq \hbar}(t_j-t_{j-1})\right]^{-1}
\end{equation*}
    and
\begin{eqnarray*}
\sum_{j=1}^{k_\ell} e^{(\rho_1-\lambda_1)(t_{k_\ell}-t_j)} (t_j-t_{j-1})
&\leq&\sum_{j=1}^{k_\ell} \int_{t_{j-1}}^{t_j} e^{(\rho_1-\lambda_1)(t_{k_\ell}-s)}\,\mathrm ds\\
&=&\int_0^{t_{k_\ell}} e^{(\rho_1-\lambda_1)(t_{k_\ell}-s)}\,\mathrm ds\leq
\frac{e^{(\rho_1-\lambda_1)t_{k\ell}}}{\rho_1-\lambda_1},
\end{eqnarray*}
 by (\ref{theorem-1}), we obtain that
\begin{equation}\label{theorem-2}
\ell e^{(\rho_1-\lambda_1)t_{k_\ell}}\leq \max_{1\leq j\leq \hbar}
\|Q_j^\top\|_{\mathbb{R}^{m\times n}}\frac{e^{(\rho_1-\lambda_1)t_{k\ell}}}{\rho_1-\lambda_1}
\left[\min_{1\leq j\leq \hbar}(t_j-t_{j-1})\right]^{-1}+\varepsilon_0.
\end{equation}
 Thus,
\begin{equation*}
\ell\leq \max_{1\leq j\leq \hbar}
\|Q_j^\top\|_{\mathbb{R}^{m\times n}}(\rho_1-\lambda_1)^{-1}
\left[\min_{1\leq j\leq \hbar}(t_j-t_{j-1})\right]^{-1}+\varepsilon_0.
\end{equation*}
Passing to the limit for $\ell\rightarrow +\infty$, we arrive at a contradiction and the result follows.

\vskip 5pt

\emph{Case 2. $\widehat{\eta}\not= 0$.}  In this case, we first note that
\begin{equation}\label{theorem-3}
    \|e^{-P^\top t}\widehat{\eta}\|_{\mathbb{R}^n}\leq e^{-\rho_1 t}\|\eta\|_{\mathbb{C}^n}=e^{-\rho_1 t}
    \;\;\mbox{for each}\;\;t\in\mathbb{R}^+.
\end{equation}
Since $[A,{\mathcal B}_\hbar,\Lambda_\hbar]$ is $(\varepsilon_0$-$\text{GCAC})_{\hbar}$,
for each $\ell\geq 2\varepsilon_0/\|\hat{\eta}\|_{\mathbb{R}^n}$, there exists a control
$(u_j^{(\ell)})_{j\in\mathbb{N}^+}\in\mathcal{U}, k_\ell\in\mathbb{N}^+$ and
$f_\ell\in B_{\varepsilon_0}(0)$ so that
\begin{equation}\label{theorem-4}
     e^{A t_{k_\ell}}\ell\widehat{\eta} e_1+\sum_{j=1}^{k_\ell}
     e^{A(t_{k_\ell}-t_j)}B_{\nu(j)}u_j^{(\ell)}=f_\ell.
\end{equation}
Recall that  $e^{At}=e^{\triangle_n t}e^{P t}=e^{P t}e^{\triangle_n t}$ for each $t\in \mathbb{R}^+$.
It follows from (\ref{theorem-4}) that
\begin{equation*}
e^{-\lambda_1 t_{k_\ell}}e^{P t_{k_\ell}}\ell\widehat{\eta}e_1+
\sum_{j=1}^{k_\ell} e^{P(t_{k_\ell}-t_j)}e^{\triangle_n (t_{k_\ell}-t_j)} B_{\nu(j)} u_j^{(\ell)}=f_\ell.
\end{equation*}
This implies that
\begin{equation*}
\ell\widehat{\eta}e_1+
e^{\lambda_1 t_{k_\ell}}\sum_{j=1}^{k_\ell} e^{-P t_j} e^{\triangle_n (t_{k_\ell}-t_j)} B_{\nu(j)} u_j^{(\ell)}
=e^{\lambda_1 t_{k_\ell}} e^{-P t_{k_\ell}}f_\ell.
\end{equation*}
Hence,
\begin{eqnarray*}
&&\ell\|\widehat{\eta}\|^2_{\mathbb{R}^n}=\langle \ell\widehat{\eta} e_1,\widehat{\eta} e_1\rangle_{(L^2(\Omega))^n}\\
&=&\langle \widehat{\eta}e_1,e^{\lambda_1 t_{k_\ell}}e^{-P t_{k_\ell}}f_\ell\rangle_{(L^2(\Omega))^n}
-\left\langle \widehat{\eta}e_1,e^{\lambda_1 t_{k_\ell}}\sum_{j=1}^{k_\ell} e^{-Pt_j}e^{\triangle_n(t_{k_\ell}-t_j)}
B_{\nu(j)}u_j^{(\ell)}\right\rangle_{(L^2(\Omega))^n}\\
&=&e^{\lambda_1 t_{k_\ell}}\langle e^{-P^\top t_{k_\ell}}\widehat{\eta}e_1,f_\ell\rangle_{(L^2(\Omega))^n}
-e^{\lambda_1 t_{k_\ell}}\sum_{j=1}^{k_\ell}\langle e^{-P^\top t_j}\widehat{\eta} e^{\triangle (t_{k_\ell}-t_j)}e_1,
\chi_{\omega_{\nu(j)}}Q_{\nu(j)}u_j^{(\ell)}\rangle_{(L^2(\Omega))^n},
\end{eqnarray*}
which, combined with (\ref{theorem-3}), indicates that
\begin{eqnarray}\label{theorem-5}
\ell\|\widehat{\eta}\|^2_{\mathbb{R}^n}
&\leq& e^{-(\rho_1-\lambda_1)t_{k_\ell}}\varepsilon_0+
\displaystyle{\sum_{j=1}^{k_\ell}} e^{-(\rho_1-\lambda_1)t_j}\|Q_{\nu(j)}\|_{\mathbb{R}^{n\times m}}\nonumber\\
&\leq& e^{-(\rho_1-\lambda_1)t_{k_\ell}}\varepsilon_0+
\displaystyle{\max_{1\leq j\leq \hbar}}\|Q_j\|_{\mathbb{R}^{n\times m}}
\displaystyle{\sum_{j=1}^{k_\ell}} e^{-(\rho_1-\lambda_1)t_j}.
\end{eqnarray}
By similar arguments as those to show (\ref{theorem-2}), we obtain from (\ref{theorem-5}) that
\begin{eqnarray*}
\ell\|\widehat{\eta}\|^2_{\mathbb{R}^n}
&\leq& e^{-(\rho_1-\lambda_1)t_{k_\ell}}\varepsilon_0+
\displaystyle{\max_{1\leq j\leq \hbar}}\|Q_j\|_{\mathbb{R}^{n\times m}}
(\rho_1-\lambda_1)^{-1}\left[\min_{1\leq j\leq \hbar}(t_j-t_{j-1})\right]^{-1}\\
&\leq&\varepsilon_0+\displaystyle{\max_{1\leq j\leq \hbar}}\|Q_j\|_{\mathbb{R}^{n\times m}}
(\rho_1-\lambda_1)^{-1}\left[\min_{1\leq j\leq \hbar}(t_j-t_{j-1})\right]^{-1}.
\end{eqnarray*}
Passing to the limit for $\ell\rightarrow +\infty$ in the latter inequality, we arrive at a
contradiction and the result follows.
\par
In summary, we finish the proof of Theorem~\ref{20-2-28}.

\section{Proof of Theorem~\ref{Intro-7}}

Before giving the proof of Theorem~\ref{Intro-7}, we first introduce the following notations.
For any fixed $\widetilde{P}\in\mathbb{R}^{n\times n}$
and $\widetilde{Q}\in\mathbb{R}^{n\times m}$, we let
\begin{equation*}
d_{\widetilde{P}}:=\min\{\pi/|\mbox{Im}\,\lambda| :  \lambda\in\sigma(\widetilde{P})\}
\end{equation*}
and
\begin{equation*}
q(\widetilde{P},\widetilde{Q}):=\max\{\mbox{dim}\mathcal{V}_{\widetilde{P}}(\widetilde{q}):
\widetilde{q}\;\;\mbox{is a column of}\;\;\widetilde{Q}\},
\end{equation*}
where $\mathcal{V}_{\widetilde{P}}(\widetilde{q}):=\mbox{span}\{\widetilde{q},  \widetilde{P}\widetilde{q}, \cdots, \widetilde{P}^{n-1}\widetilde{q}\}$
and we agree that $\frac{1}{0}=+\infty$.
The following result can be found in the proof of Theorem 2.2 in \cite{Q-W}.
\begin{lemma}\label{theo-1} Let $\widetilde{P}\in \mathbb{R}^{n\times n}$ and
$\widetilde{Q}\in \mathbb{R}^{n\times m}$. For each increasing strictly sequence
$\{\tau_i\}_{i=1}^{q(\widetilde{P},\widetilde{Q})}$ with
$\tau_{q(\widetilde{P},\widetilde{Q})}-\tau_1<d_{\widetilde{P}}$,
\begin{equation*}
    \mbox{span}\;\{e^{\widetilde{P}\tau_1}\widetilde{Q}, e^{\widetilde{P}\tau_2}\widetilde{Q},
    \cdots, e^{\widetilde{P}\tau_{q(\widetilde{P},\widetilde{Q})}}\widetilde{Q}\}
    =\mbox{span}\;\{\widetilde{Q}, \widetilde{P}\widetilde{Q}, \cdots, \widetilde{P}^{n-1}\widetilde{Q}\}.
\end{equation*}
\end{lemma}

\vskip 5pt
\begin{proof}[{Proof of Theorem~\ref{Intro-7}}] Since rank$(\lambda I_n-P, Q_1, Q_2, \cdots, Q_\hbar)=n$ for each $\lambda\in \mathbb{C}$, by Lemma 3.3.7 in \cite{Sontag},
we have that
\begin{equation}\label{theo-2}
\mbox{rank}(Q_1, \cdots, Q_\hbar, P Q_1, \cdots, P Q_\hbar, \cdots, P^{n-1} Q_1, \cdots, P^{n-1} Q_\hbar)
=n.
\end{equation}
Set
\begin{equation}\label{theo-3}
q_j:=q(-P,Q_{j})-1\;\;\mbox{for each}\;\; 1\leq j\leq \hbar\;\;\mbox{and}\;\; q:=\max_{1\leq j\leq \hbar} q_j.
\end{equation}
Let  $t_\hbar$ be a fixed positive constant satisfying that
\begin{equation}\label{theo-4}
q t_\hbar<d_{-P}.
\end{equation}
We arbitrarily choose $\{t_j\}_{j=0}^{\hbar-1}$ with $0=t_0<t_1<\cdots<t_{\hbar-1}<t_\hbar$.
For each $\ell\in\mathbb N^+$, we define $\{t_{j+\ell\hbar}\}_{j=1}^{\hbar}$ as follows:
\begin{equation}\label{theo-5}
   t_{j+\ell\hbar}:=t_j+\ell t_\hbar\;\;\mbox{for all}\;\;1\leq j\leq \hbar.
\end{equation}
It is clear that $\Lambda_\hbar:=\{t_j\}_{j\in\mathbb N}\in {\mathcal M}_\hbar.$

We next claim that the system $[A,\mathcal{B}_\hbar,\Lambda_\hbar]$ is $(\text{GCAC})_\hbar$.
Indeed, for each $1\leq j\leq \hbar$, by (\ref{theo-3})-(\ref{theo-5}), we get that
\begin{equation}\label{theo-6}
t_{j+q_j\hbar}-t_{j}\leq t_{j+q\hbar}-t_{j}=q t_{\hbar}<d_{-P}.
\end{equation}
According to Lemma~\ref{theo-1} (where $\widetilde{P}, \widetilde{Q}$ and
$\{\tau_i\}_{i=1}^{q(\widetilde{P},\widetilde{Q})}$ are replaced by $-P, Q_j$
and $\{t_{j+(i-1)\hbar}\}_{i=1}^{q(-P, Q_j)}$, respectively), it follows from (\ref{theo-3})
and (\ref{theo-6}) that
\begin{eqnarray*}
\mbox{span}\{Q_j, (-P) Q_j, \cdots, (-P)^{n-1}Q_j\}&=&
\mbox{span}\{e^{-P t_j}Q_j, e^{-P t_{j+\hbar}}Q_j, \cdots, e^{-P t_{j+q_{j\hbar}}}Q_j\}\\
&\subset&\mbox{span}\{e^{-P t_j}Q_j, e^{-P t_{j+\hbar}}Q_j,\cdots, e^{-P t_{j+q\hbar}}Q_j\}.
\end{eqnarray*}
This implies that for each $1\leq j\leq \hbar$,
\begin{eqnarray*}
\mbox{span}\{Q_j, P Q_j, \cdots, P^{n-1}Q_j\}&=&
\mbox{span}\{Q_j, (-P) Q_j, \cdots, (-P)^{n-1}Q_j\}\\
&\subset&\mbox{span}\{e^{-P t_j}Q_j, e^{-P t_{j+\hbar}}Q_j,\cdots, e^{-P t_{j+q\hbar}}Q_j\}.
\end{eqnarray*}
Hence,
\begin{eqnarray*}
&&\mbox{span}\{Q_1, \cdots, Q_\hbar, P Q_1,\cdots, P Q_\hbar, \cdots, P^{n-1} Q_1, \cdots, P^{n-1}Q_\hbar\}\\
&\subset&\mbox{span}\{e^{-P t_1}Q_1, \cdots, e^{-P t_{\hbar}}Q_\hbar,
\cdots, e^{-P t_{1+q\hbar}}Q_1, \cdots, e^{-P t_{\hbar+q\hbar}}Q_\hbar\},
\end{eqnarray*}
which, combined with (\ref{theo-2}) and (\ref{Intro-4}), indicates that
\begin{equation*}
   \mbox{rank}(e^{-P t_1}Q_{\nu(1)}, e^{-P t_2}Q_{\nu(2)},\cdots, e^{-P t_{(q+1)\hbar}}Q_{\nu((q+1)\hbar)})=n.
\end{equation*}
This, along with Theorem~\ref{Intro-6}, implies Theorem~\ref{Intro-7}.
\end{proof}

\section{Appendices}
\subsection{Appendix A}
    In this subsection, we prove that the system $[A,\mathcal{B}_{\hbar},\Lambda_{\hbar}]$ is $(\text{GCAC})_\hbar$ if and only if
    {\textbf{Claim (H)}} holds.
\par
Indeed, let $T:=\max_{0\leq k< \hbar}|t_{k+1}-t_k|$
   and $M:=\max_{0\leq t\leq T} \|e^{At}\|_{\mathcal L((L^2(\Omega))^n;(L^2(\Omega))^n)}$.
   Since $\Lambda_\hbar\in \mathcal{M}_\hbar$, we have that
   \begin{equation}\label{20-3-1}
     T=\max_{ k\in\mathbb N}|t_{k+1}-t_k|.
   \end{equation}
\par
On one hand,  if the system $[A,\mathcal{B}_{\hbar},\Lambda_{\hbar}]$ is $(\text{GCAC})_\hbar$, then {\textbf{Claim (H)}} follows directly from Definition~\ref{Intro-5}.
\par
    On the other hand, if  {\textbf{Claim (H)}} is true, then for each $\varepsilon>0$ and each $x_0\in (L^2(\Omega))^n$, there exists a control $u:=(u_j)_{j\in\mathbb{N}^+}\in \mathcal U$
  and $T_0\in\mathbb R^+$,  so that
  \begin{equation}\label{20-3-2}
  \|x(T_0;x_0,u,\Lambda_\hbar)\|_{(L^2(\Omega))^n}\leq { \varepsilon}/{M}.
  \end{equation}
  Let $k_0\in \mathbb N$ be such that $t_{k_0}\leq T_0< t_{k_0+1}$. Two cases may occur: $k_0=0$ or $k_0\neq0$.
\vskip 5pt
    \emph{Case 1. $k_0=0$.} In this case,  we observe that
  \begin{equation*}
  x(t_1;x_0,0,\Lambda_\hbar)=e^{At_1}x_0=e^{A(t_1-T_0)} e^{AT_0}x_0=e^{A(t_1-T_0)}x(T_0;x_0,u,\Lambda_\hbar).
  \end{equation*}
  This, along with \eqref{20-3-1} and \eqref{20-3-2}, implies that
  \begin{equation*}
  \|x(t_1;x_0,0,\Lambda_\hbar)\|_{(L^2(\Omega))^n}\leq M\cdot({\varepsilon}/{M})=\varepsilon.
  \end{equation*}
\vskip 5pt
    \emph{Case 2. $k_0\neq 0$.} In this case, we define $\widetilde u=(\widetilde u_j)_{j\in\mathbb N}$ as follows:
   \begin{equation*}
    \widetilde u_j=\begin{cases}
    u_j & \mbox{ if  }\;1\leq j\leq k_0,\\
    0   & \mbox{ if  }\;j\geq k_0+1.
    \end{cases}
   \end{equation*}
   It is clear that $\widetilde u\in\mathcal U$ and
   \begin{eqnarray}\label{20-3-3}
   &&x(t_{k_0+1};x_0,\widetilde u,\Lambda_\hbar)\nonumber\\
    &=& e^{A(t_{k_0+1}-T_0)}\left(e^{AT_0}x_0+e^{A (T_0-t_1)}B_1u_1+\cdots+ e^{A( T_0-t_{k_0})}B_{k_0}u_{k_0}\right)\nonumber\\
    &=&e^{A(t_{k_0+1}-T_0)}x(T_0;x_0, u,\Lambda_\hbar).
   \end{eqnarray}
  It follows from \eqref{20-3-1}-\eqref{20-3-3} that
  \begin{equation*}
   \|x(t_{k_0+1};x_0,\widetilde u,\Lambda_\hbar)\|_{(L^2(\Omega))^n}\leq M\cdot({\varepsilon}/{M})=\varepsilon.
  \end{equation*}
   In summary, the system $[A,\mathcal{B}_{\hbar},\Lambda_{\hbar}]$ is $(\text{GCAC})_\hbar$. The proof is completed.

\subsection{Appendix B}

\noindent\textit{Proof of Corollary 1.6.}
Firstly, we claim that
\begin{equation}\label{4-24-2}
 [A,\mathcal B_\hbar, \Lambda_\hbar] \textrm{ is constrained null controllable } \Rightarrow \sigma(P)\subset \{\rho\in\mathbb C:\,  Re(\rho)\leq \lambda_1\}.
\end{equation}
 Otherwise, there would exist a $\rho\in\sigma(P)$ so that $\mbox{Re}(\rho)>\lambda_1$, which, combined with Theorem \ref{20-2-28}, indicates that the system $ [A,\mathcal B_\hbar, \Lambda_\hbar]$ is not constrained null controllable. This leads to a contradiction. Hence, \eqref{4-24-2} follows.

 Next, we show that
 \begin{equation}\label{4-24-3}
 \sigma(P)\subset \{\rho\in\mathbb C:\,  \mbox{Re}(\rho)\leq \lambda_1\} \Rightarrow [A,\mathcal B_\hbar, \Lambda_\hbar] \textrm{ is constrained null controllable. }
 \end{equation}
Its proof is split into the following three steps.
\vskip 5pt
\emph{Step 1. We prove that there is a constant $C(k^*)>0$ so that
\begin{equation}\label{20-4-23-1}
\|e^{A^*t_{k^*}}z\|^2_{(L^2(\Omega))^n}\leq C(k^*)\sum_{j=1}^{k^*}\|B_{\nu(j)}^* e^{A^*(t_{k^*}-t_j)}z\|_{(L^2(\Omega))^m}^2 \textrm{ for each }z\in (L^2(\Omega))^n.
\end{equation}}

Indeed,  according to  \eqref{yu-3-6-2} and Lemma \ref{Preli-11}, there is a constant $C(k^*)>0$ so that
\begin{equation*}
\|v\|_{\mathbb R^n}^2 \leq C(k^*) \sum_{j=1}^{k^*} \|Q_{\nu(j)}^\top e^{-P^\top t_j}v\|_{\mathbb R^m}^2 \textrm{ for each }v\in \mathbb R^n.
\end{equation*}
This, along with  \eqref{Preli-2:3} and \eqref{20-2-13},  yields  that for each $z\in (L^2(\Omega))^n,$
\begin{eqnarray}\label{4-24-1}
\|e^{A^*t_{k^*}}z\|^2_{(L^2(\Omega))^n}&=&\|e^{\Delta_n t_{k^*}}e^{P^\top t_{k^*}}z\|^2_{(L^2(\Omega))^n}\nonumber\\
        &\leq& C(k^*)\sum_{j=1}^{k^*} \|Q_{\nu(j)}^\top e^{-P^\top t_j}e^{\Delta_n t_{k^*}} e^{P^{\top} t_{k^*}}z\|^2_{(L^2(\Omega))^m}\\
&\leq& C(k^*)\sum_{j=1}^{k^*} \|Q_{\nu(j)}^\top e^{(\Delta_n+P^\top)(t_{k^*}-t_j)}z\|^2_{(L^2(\Omega))^m}\nonumber
\end{eqnarray}
Here, we used the fact that $\|e^{\Delta_m t}\|_{\mathcal L((L^2(\Omega))^m;(L^2(\Omega))^m)}\leq 1$ for each $t\in\mathbb R^+$.
Since $\Omega=\omega=\cap_{k=1}^\hbar \omega_k$,  \eqref{20-4-23-1} follows from  \eqref{4-24-1} and \eqref{Intro-4(1)} immediately.

\vskip 5pt
\emph{Step 2.  We show that
for each $ x_0\in (L^2(\Omega))^n$, there is a control
$ u:=( u_j)_{j\in\mathbb N^+}\in l^2(\mathbb N^+;(L^2(\Omega))^m)$   so that
\begin{equation}\label{4-24-5}
x(t_{k^*}; x_0,u,\Lambda_\hbar)=0
\textrm{ and }
\| u\|_{l^2(\mathbb N^+;(L^2(\Omega))^m)}\leq  \sqrt{C(k^*)}\| x_0\|_{(L^2(\Omega))^n},
\end{equation}
where $C(k^*)$ is the constant given in \eqref{20-4-23-1}.}

To this end, we set
\begin{equation*}
 X:=\left\{ \left(B^*_{\nu(1)}e^{A^*(t_{k^*}-t_1)}z, \cdots, B^*_{\nu(k^*-1)}e^{A^*(t_{k^*}-t_{k^*-1})}z, B^*_{\nu(k^*)}z\right): z\in (L^2(\Omega))^n\right\}
\end{equation*}
and
\begin{equation*}
Y:=\underbrace{(L^2(\Omega))^m\times\cdots\times (L^2(\Omega))^m}_{k^*}.
\end{equation*}
It is clear that $X$ is a linear subspace of $Y$. We define a linear functional $F: X\to \mathbb R$ by setting
\begin{equation}\label{4-29-1}
F\left(\big(B^*_{\nu(1)}e^{A^*(t_{k^*}-t_1)}z, \cdots, B^*_{\nu(k^*-1)}e^{A^*(t_{k^*}-t_{k^*-1})}z, B^*_{\nu(k^*)}z\big)\right):=-\langle x_0, e^{A^*t_{k^*}}z \rangle_{(L^2(\Omega))^n}.
\end{equation}
It follows from \eqref{20-4-23-1} that $F$ is well defined and
\begin{equation*}
\begin{split}
&\left|F\left(\big(B^*_{\nu(1)}e^{A^*(t_{k^*}-t_1)}z, \cdots, B^*_{\nu(k^*-1)}e^{A^*(t_{k^*}-t_{k^*-1})}z, B^*_{\nu(k^*)}z\big)\right)\right|\\
\leq\quad &\sqrt{C(k^*)}\|x_0\|_{(L^2(\Omega))^n} \sqrt{\sum_{j=1}^{k^*}\|B^*_{\nu(j)}e^{A^*(t_{k^*}-t_j)}z\|^2_{(L^2(\Omega))^m}}.
\end{split}
\end{equation*}
Hence, $F$ is a linear bounded functional on $X$ and
\begin{equation*}
\|F\|_{\mathcal L(X;\mathbb R)}\leq \sqrt{C(k^*)}\|x_0\|_{(L^2(\Omega))^n}.
\end{equation*}
According to the Hahn-Banach theorem, there is a linear bounded functional $G:\,Y\to \mathbb R$ so that
$G=F$ on $X$ and
\begin{equation*}
\|G\|_{\mathcal L(Y;\mathbb R)}=\|F\|_{\mathcal L(X;\mathbb R)}\leq \sqrt{C(k^*)}\|x_0\|_{(L^2(\Omega))^n}.
\end{equation*}
Then by the Riesz representation theorem, there exists a function $(g_1, g_2, \cdots, g_{k^*})\in Y$ so that for any
$(v_1, v_2, \cdots, v_{k^*})\in Y$,
\begin{equation*}
G\left((v_1,v_2,\cdots, v_{k^*})\right)= \sum_{j=1}^{k^*} \langle g_j,v_j\rangle_{(L^2(\Omega))^m}
\end{equation*}
 and
\begin{equation}\label{4-29-3}
\sum_{j=1}^{k^*}\|g_j\|_{(L^2(\Omega))^m}^2=\|G\|^2_{\mathcal L(Y;\mathbb R)}\leq C(k^*)\|x_0\|^2_{(L^2(\Omega))^n}.
\end{equation}
Thus,
\begin{equation*}
F\left(\big(B^*_{\nu(1)}e^{A^*(t_{k^*}-t_1)}z, \cdots, B^*_{\nu(k^*-1)}e^{A^*(t_{k^*}-t_{k^*-1})}z, B^*_{\nu(k^*)}z\big)\right)\\
= \sum_{j=1}^{k^*}\langle g_j, B^*_{\nu(j)}e^{A^*(t_{k^*}-t_j)}z\rangle_{(L^2(\Omega))^m}.
\end{equation*}
This, along with \eqref{4-29-1}, implies that
\begin{equation*}
\left\langle e^{At_{k^*}}x_0+\sum_{j=1}^{k^*}e^{A(t_{k^*}-t_{j})}B_{\nu(j)}g_j, z\right\rangle_{(L^2(\Omega))^n}=0 \textrm{ \;\; for each }z\in {(L^2(\Omega))^n},
\end{equation*}
which indicates that
\begin{equation}\label{4-29-4}
e^{At_{k^*}}x_0+\sum_{j=1}^{k^*}e^{A(t_{k^*}-t_{j})}B_{\nu(j)}g_j=0.
\end{equation}
Define $u:=(u_j)_{j\in\mathbb N^+}$ with
\begin{equation}\label{4-29-5}
u_j:=\begin{cases}
g_j& \textrm{if}\;\; 1\leq j\leq k^*,\\
0& \textrm{if}\;\;  j\geq  k^*+1.
\end{cases}
\end{equation}
 Then \eqref{4-24-5}  follows from  \eqref{4-29-3}-\eqref{4-29-5} immediately.

\vskip 5pt

\emph{Step 3.  We finish the proof of \eqref{4-24-3}.}

For each $x_0\in (L^2(\Omega))^n$, we fix it. Let
\begin{equation}\label{4-23-4}
\varepsilon:=(M\sqrt{C(k^*)})^{-1}
\textrm{ and }
M:=\max\left\{\sum_{0\leq k\leq \hbar-1}\|e^{A(t_\hbar-t_k)}\|_{\mathcal L((L^2(\Omega))^n;(L^2(\Omega))^n)},1\right\}.
\end{equation}
By $(i)$ in Theorem \ref{Intro-6}, we have that the system $[A,\mathcal B_\hbar, \Lambda_\hbar]$ is $\textrm{(GCAC)}_\hbar$.
Hence, there is  a $\hat k\in \mathbb N$
and a control $\hat v:=(\hat v_j)_{j\in\mathbb N^+}\in\mathcal U$ so that
\begin{equation}\label{4-23-7}
x(t_{\hat k};x_0,\hat v,\Lambda_\hbar)\in B_\varepsilon(0).
\end{equation}
Two cases may occur: $\hat k=0$ or $\hat k\neq 0$.

\textit{Case 1.} $\hat k=0$. In this case, $x_0\in B_\varepsilon(0)$. On one hand, according to \eqref{4-24-5},
there is a control  $u:=(u_j)_{j\in\mathbb N^+}\in l^2(\mathbb N^+;(L^2(\Omega))^m)$ so that
\begin{equation}\label{4-23-6}
x(t_{k^*};x_0,u,\Lambda_\hbar)=0,
\end{equation}
and
\begin{equation}\label{4-23-5}
\|u_j\|_{(L^2(\Omega))^m}\leq \|u\|_{l^2(\mathbb N^+;(L^2(\Omega))^m)}\leq  \sqrt{C(k^*)}\varepsilon
\textrm{ for each } j\in \mathbb N^+.
\end{equation}
On the other hand, it follows from \eqref{4-23-4}  that
$
\sqrt{C(k^*)}\varepsilon= M^{-1}\leq 1.
$
This, along with \eqref{4-23-5}, implies that
\begin{equation*}
\|u_j\|_{(L^2(\Omega))^m}\leq 1 \textrm{ for each } j\in\mathbb N^+,
\end{equation*}
which indicates   $u\in\mathcal U$. Thus, by \eqref{4-23-6},  we get that the system $[A,\mathcal B_\hbar, \Lambda_\hbar]$
is constrained null controllable.

\textit{Case 2. } $\hat k\neq 0$. In this case, we take a $\gamma\in\mathbb N$ so that $\gamma \hbar \leq \hat k<(\gamma+1)\hbar$. Let
\begin{equation}\label{4-23-8}
\widetilde v_j:=\begin{cases}
\hat v_j &\textrm{if}\;\;1\leq j\leq \hat k,\\
0 & \textrm{if}\;\;j\geq \hat k+1,
\end{cases}
\end{equation}
and $\widetilde v:=(\widetilde v_j)_{j\in \mathbb N^+}.$ Since $\hat v\in \mathcal U$, it is clear that
\begin{equation}\label{4-23-11}
\widetilde v\in \mathcal U.
\end{equation}
On one hand, by \eqref{Intro-2},  \eqref{4-23-8}, \eqref{4-23-7} and the second definition of \eqref{4-23-4}, we obtain that
\begin{eqnarray}\label{4-23-9}
\|x((\gamma+1)\hbar;x_0,\widetilde v,\Lambda_\hbar)\|_{(L^2(\Omega))^n}
&=&\| e^{A(t_{(\gamma+1)\hbar}-t_{\hat k})}x(t_{\hat k};x_0,\hat v,\Lambda_\hbar)\|_{(L^2(\Omega))^n}\\
&=&\|e^{A(t_{\hbar}-t_{\hat k-\gamma \hbar})}x(t_{\hat k};x_0,\hat v,\Lambda_\hbar)\|_{(L^2(\Omega))^n}\leq M\varepsilon.\nonumber
\end{eqnarray}
Let $\hat x_0:= x((\gamma+1)\hbar;x_0,\widetilde v,\Lambda_\hbar).$ By \eqref{4-24-5} and \eqref{4-23-9},
there is a control $\hat u:=(\hat u_j)_{j\in\mathbb N^+}\in l^2(\mathbb{N}^+;(L^2(\Omega))^m)$ so that
\begin{equation}\label{4-23-13}
x(t_{k^*};\hat x_0,\hat u,\Lambda_\hbar)=0
\end{equation}
and
\begin{equation}\label{4-23-10}
\|\hat u_j\|_{(L^2(\Omega))^m}\leq \|\hat u\|_{l^2(\mathbb N^+;(L^2(\Omega))^m)}
\leq \sqrt{C(k^*)} M\varepsilon \textrm{ \;\;for each }j\in\mathbb N^+.
\end{equation}
On the other hand, it follows from  \eqref{4-23-10} and the first definition of  \eqref{4-23-4} that
\begin{equation}\label{4-23-12}
\|\hat u_j\|_{(L^2(\Omega))^m}\leq 1 \textrm{ for each }j\in\mathbb N^+.
\end{equation}
Define
\begin{equation*}
u_j:=\begin{cases}
\widetilde  v_j &\textrm{if}\;\;1\leq j\leq (\gamma+1)\hbar,\\
\hat u_{j-(\gamma+1)\hbar}&\textrm{if}\;\;j\geq (\gamma+1)\hbar+1,
\end{cases}
\end{equation*}
and $u:=(u_j)_{j\in\mathbb N^+}$. According to  \eqref{4-23-11} and \eqref{4-23-12},  $u\in \mathcal U$. Moreover, by
\eqref{4-23-13} and \eqref{Intro-2},  we observe that
\begin{eqnarray*}
&&x(t_{(\gamma+1)\hbar+k^*};x_0,u,\Lambda_\hbar)\\
&=&x(t_{(\gamma+1)\hbar+k^*}-t_{(\gamma+1)\hbar};x(t_{(\gamma+1)\hbar};x_0,\widetilde v, \Lambda_\hbar), \hat u,\Lambda_\hbar)\\
&=&x(t_{k^*};\hat x_0,\hat u,\Lambda_\hbar)=0.
\end{eqnarray*}
This implies that the system $[A,\mathcal B_\hbar,\Lambda_\hbar]$ is constrained null controllable.

In summary,  we finish the proof of \eqref{4-24-3}.

\begin{remark}
We observe that if the claim of \emph{Step 2} in the proof of Corollary
 \ref{yu-proposition-b-1} holds, then \eqref{20-4-23-1} is also true. Indeed, according to \eqref{4-24-5},
there exists $(u_1,\cdots,u_{k^*})\in Y$ so that
\begin{equation*}
\sum_{j=1}^{k^*}\|u_j\|_{(L^2(\Omega))^m}^2\leq C(k^*)\|x_0\|^2_{(L^2(\Omega))^n}
\end{equation*}
and
\begin{equation*}
e^{At_{k^*}}x_0+\sum_{j=1}^{k^*}e^{A(t_{k^*}-t_{j})}B_{\nu(j)}u_j=0.
\end{equation*}
It follows from the above equality that
\begin{equation*}
\langle e^{A^*t_{k^*}} z, x_0\rangle_{(L^2(\Omega))^n} =-\sum_{j=1}^{k^*}\langle B_{\nu(j)}u_j, e^{A^*(t_{k^*}-t_{j})}z\rangle_{(L^2(\Omega))^n}\textrm{\;\; for each }z\in (L^2(\Omega))^n.
\end{equation*}
We choose $x_0=e^{A^*t_{k^*}} z$. Then
\begin{eqnarray*}
\|e^{A^*t_{k^*}}z\|^2_{(L^2(\Omega))^n}
&=&-\sum_{j=1}^{k^*}\langle u_j, B_{\nu(j)}^*e^{A^*(t_{k^*}-t_{j})}z\rangle_{(L^2(\Omega))^m}\\
&\leq& \sqrt{\sum_{j=1}^{k^*}\|u_j\|^2_{(L^2(\Omega))^m}}\sqrt{\sum_{j=1}^{k^*}\|B^*_{\nu(j)}e^{A^*(t_{k^*}-t_j)}z\|^2_{(L^2(\Omega))^m}}\\
&\leq& \sqrt{ C(k^*)}\|e^{A^*t_{k^*}}z\|_{(L^2(\Omega))^n}\sqrt{\sum_{j=1}^{k^*}\|B^*_{\nu(j)}e^{A^*(t_{k^*}-t_j)}z\|^2_{(L^2(\Omega))^m}}.
\end{eqnarray*}
This implies \eqref{20-4-23-1} immediately.
\end{remark}

\subsection{Appendix C}
%
%
%
%
%

\noindent {\bf Example.}
Let $\hbar:=1$, $t_k:=k$ for all $k\in \mathbb N$ and $\omega_1\subset\Omega$. Then $\Lambda_\hbar:=\{t_k\}_{k\in\mathbb N}\in \mathcal M_\hbar(=\mathcal{M}_1)$ and $\omega_{\nu(k)}=\omega_1$ for each $k\in\mathbb{N}^+$.
Let $n:=2$, $m:=1$, $P:=\lambda_1 I_2$ and $Q_1:=\left(
\begin{array}{c}
 0\\
1
\end{array}
\right)$.
    It is clear that $\langle P\eta,\eta\rangle_{\mathbb{R}^2}= \lambda_1\|\eta\|^2_{\mathbb{R}^2}$ for each $\eta\in\mathbb{R}^2$,  $\sigma(P)=\{\lambda_1\}$ and
\begin{equation*}
e^{-Pt_k}Q_{\nu(k)}
=e^{-Pt_k}Q_1=e^{-\lambda_1 t_k} \left(
\begin{array}{c}
 0\\
1
\end{array}
\right)
\mbox{\;\;for each } k\in\mathbb N^+.
\end{equation*}
Hence, \eqref{yu-3-6-2} does not hold.

Now we claim that the system  $[A,\mathcal B_\hbar,\Lambda_\hbar]$  is not $(\text{GCAC})_\hbar$.
To this end, for any fixed $\varepsilon>0$, we let $x_0:=2\varepsilon (1,0)^\top e_1\in (L^2(\Omega))^2$. For each $k\in \mathbb N^+$ and $u=(u_j)_{j\in\mathbb N^+}\in \mathcal{U}$, we observe that
\begin{equation*}
\begin{split}
e^{A(t_k-t_j) }\chi_{\omega_{\nu(j)}}Q_{\nu(j)} u_j\,=\;\;&e^{\triangle (t_k-t_j)}e^{P(t_k-t_j)}\chi_{\omega_1}Q_1u_j\\
=\;\;&e^{\lambda_1(t_k-t_j)}e^{\triangle(t_k-t_j)} \chi_{\omega_1}u_j
\left(
\begin{array}{c}
0\\
1
\end{array}
\right),
\end{split}\;\;\mbox{for each}\;\;1\leq j\leq k,
\end{equation*}
 and
\begin{equation*}
e^{At_k}x_0
=2\varepsilon e^{\triangle t_k}e_1e^{Pt_k}
\left(\begin{array}{c}
1\\
0
\end{array} \right)
=2\varepsilon e_1
\left(\begin{array}{c}
1\\
0
\end{array} \right).
\end{equation*}
   These imply that
\begin{eqnarray*}
\left\|x(t_k;x_0,u,\Lambda_\hbar)\right\|_{(L^2(\Omega))^2}
&=&\big\|e^{At_k}x_0+\sum_{j=1}^k e^{A(t_k-t_j)}\chi_{\omega_{\nu(j)}}Q_{\nu(j)} u_j\big\|_{(L^2(\Omega))^2}\\
&=&\Big\|2\varepsilon e_1
\left(\begin{array}{c}
1\\
0
\end{array} \right)
+
\left(\sum_{j=1}^ke^{\lambda_1(t_k-t_j)}e^{\triangle(t_k-t_j)}\chi_{\omega_1} u_j\right)
\left(
\begin{array}{c}
0\\
1
\end{array}
\right)
\Big\|_{(L^2(\Omega))^2}
\geq 2\varepsilon.
\end{eqnarray*}
  From the latter, it follows that the system $[A,\mathcal B_\hbar,\Lambda_\hbar]$ is not $(\varepsilon$-$\text{GCAC})_\hbar$ by the arbitrariness of $k$ and $u$.
    Thus, the system  $[A,\mathcal B_\hbar,\Lambda_\hbar]$  is not $(\text{GCAC})_\hbar$.
\subsection{Appendix D}

\begin{proof}[Proof of Lemma \ref{Preli-11}]

Firstly, we prove $(i)\Rightarrow(ii)$.
By contradiction, if (\ref{20-2-12}) were not true, then for each $i\in\mathbb{N}^+$,
there would exist a $v_i\in \mathbb{R}^n$ with $\|v_i\|_{\mathbb{R}^n}=1$ so that
\begin{equation}\label{Preli-2:7}
\|v_i\|^2_{\mathbb{R}^n}> i\sum_{j=1}^k \|\widetilde Q_{j}^\top e^{-\widetilde P^{\top}\tau_j}v_i\|^2_{\mathbb{R}^m}.
\end{equation}
Since $\|v_i\|_{\mathbb{R}^n}=1$ for each $i\in\mathbb{N}^+$, there is a subsequence of $\{v_i\}_{i\in\mathbb{N}^+}$,
still denoted by itself, and $v^*\in \mathbb{R}^n$ with $\|v^*\|_{\mathbb{R}^n}=1$
so that $v_i\to v^*$ as $i\to +\infty$. Then
$$
\sum_{j=1}^k \|\widetilde Q_{j}^\top e^{-\widetilde P^{\top}\tau_j}v_i\|^2_{\mathbb{R}^m}\to
 \sum_{j=1}^k \|\widetilde Q_{j}^\top e^{-\widetilde P^{\top}\tau_j}v^*\|^2_{\mathbb{R}^m}\;\;\mbox{as}\;\;i\to +\infty.
$$
This, along with (\ref{Preli-2:7}), implies that
$$
\sum_{j=1}^k \|\widetilde Q_{j}^\top e^{-\widetilde P^{\top}\tau_j}v^*\|^2_{\mathbb{R}^m}=0,
$$
which, combined with $(i)$, indicates that $v^*=0$. It contradicts to the fact that
$\|v^*\|_{\mathbb{R}^n}=1$. Hence, (\ref{20-2-12}) holds.

Next, we show that $(ii)\Rightarrow(i)$. Otherwise,  there would exist a $v^*\in\mathbb R^n\backslash\{0\}$ so that
$$
\widetilde Q_j^\top e^{-\widetilde P^\top \tau_j}v^*=0\, \textrm{ for each }1\leq j\leq k.
$$
It follows from this and $(ii)$ that
$
\|v^*\|_{\mathbb R^n}=0,
$
which  contradicts to the fact that $v^*\neq 0$.

This completes the proof.
\end{proof}

\end{document}